\documentclass[10pt]{amsart}

\usepackage{amssymb}
\usepackage{amscd}
\usepackage{multirow}
\usepackage{stmaryrd}
\usepackage{mathrsfs}
\usepackage{array,subfigure}
\usepackage{graphicx}

\newtheorem{thm}{Theorem}[section]
\newtheorem{prop}[thm]{Proposition}
\newtheorem{cor}[thm]{Corollary}
\newtheorem{lem}[thm]{Lemma}
\newtheorem{defn}[thm]{Definition}

\newenvironment{xpl}{\refstepcounter{thm} \medskip \noindent {\bf  Example \arabic{section}.\arabic{thm}}}{\hfill$\diamondsuit$\mbox{}\bigskip}

\newcounter{num}

\newenvironment{thmlist}{\begin{list}{(\roman{num})}{\usecounter{num}\setlength{\leftmargin}{25pt}
\setlength{\itemindent}{0pt}\setlength{\labelwidth}{20pt}\setlength{\labelsep}{5pt}\setlength{\itemsep}{0in}}}{\end{list}}

\newcommand{\C}{\mathbb{C}}

\newcommand{\R}{\mathbb{R}}
\newcommand{\Z}{\mathbb{Z}}
\newcommand{\N}{\mathbb{N}}
\newcommand{\Q}{\mathbb{Q}}
\newcommand{\Cl}{\operatorname{Cl}}
\newcommand{\cps}{\mathbb{C}P}

\newcommand{\contr}{\,\lrcorner\,}
\newcommand{\ol}[1]{\bar{#1}}

\newcommand{\superscript}[1]{\ensuremath{^{\textrm{#1}}}}
\newcommand{\codim}{\operatorname{codim}}

\newcommand{\Aut}{\operatorname{Aut}}
\newcommand{\Hom}{\operatorname{Hom}}

\newcommand{\ric}{\operatorname{Ricci}}
\newcommand{\Ric}{\operatorname{Ric}}
\newcommand{\Ord}{\operatorname{Ord}}
\newcommand{\Ind}{\operatorname{Ind}}
\newcommand{\Vol}{\operatorname{Vol}}

\newcommand{\rank}{\operatorname{rank}}
\newcommand{\Spec}{\operatorname{Spec}}
\newcommand{\SF}{\operatorname{SF}}

\newcommand{\Picorb}{\operatorname{Pic^{\text{orb}}}}
\newcommand{\Pic}{\operatorname{Pic}}

\newcommand{\CDiv}{\operatorname{CDiv}}
\newcommand{\WDiv}{\operatorname{WDiv}}

\newcommand{\Sing}{\operatorname{Sing}}
\newcommand{\Int}{\operatorname{Int}}

\title{Examples of asymptotically conical Ricci-flat K\"{a}hler manifolds}
\author{Craig van Coevering}
\address{Department of Mathematics, Massachusetts Institute of Technology, 77 Massachusetts Avenue, Cambridge, MA 02139-4307}
\email{craig@math.mit.edu}
\date{December 30, 2008}

\begin{document}

\begin{abstract}
Previously the author has proved that a crepant resolution $\pi:Y\rightarrow X$ of a Ricci-flat K\"{a}hler cone $X$
admits a complete Ricci-flat K\"{a}hler metric asymptotic to the cone metric in every K\"{a}hler class in $H^2_c(Y,\R)$.
These manifolds can be considered to be generalizations of the Ricci-flat ALE K\"{a}hler spaces known by the work of
P. Kronheimer, D. Joyce and others.

This article considers further the problem of constructing examples.  We show that every 3-dimensional Gorenstein toric K\"{a}hler
cone admits a crepant resolution for which the above theorem applies.  This gives infinitely many examples of asymptotically conical
Ricci-flat manifolds.    Then other examples are given of which are crepant
resolutions hypersurface singularities which are known to admit Ricci-flat K\"{a}hler cone metrics by the work of
C. Boyer, K. Galicki, J. Koll\'{a}r, and others.  We concentrate on 3-dimensional examples.  Two families of
hypersurface examples are given which are distinguished by the condition $b_3(Y)=0$ or $b_3(Y)\neq 0$.
\end{abstract}

\maketitle

\section{Introduction}

Recall that a Sasaki-Einstein manifold is a Riemannian manifold $(S,g)$ whose metric cone
$(C(S),\ol{g})$, $C(S)=\R_{>0} \times S$ and $\ol{g}=dr^2 +r^2 g$, is Ricci-flat K\"{a}hler.
It follows that $(S,g)$ is positive scalar curvature Einstein.  Besides the case $S=S^{2n-1}$ and $C(S)=\C^n$
the K\"{a}hler cone $C(S)$ has a singularity at the apex.  In~\cite{vC2} the author investigated the existence
of a Ricci-flat K\"{a}hler metric on a resolution $\pi:Y\rightarrow X$, where $X=C(S)\cup\{o\}$ is a Ricci-flat
K\"{a}hler cone.  The resolution will necessarily be a crepant, and one requires that the metric on $Y$ be
asymptotic to the original Ricci-flat K\"{a}hler cone metric on $X$.  The following theorem was proved.

\begin{thm}[\cite{vC2}]\label{thm:main}
Let $\pi :Y\rightarrow X$ be a crepant resolution of the isolated singularity of $X=C(S)$, where $C(S)$ admits a
Ricci-flat K\"{a}hler cone metric.  Then $Y$ admits a Ricci-flat K\"{a}hler metric $g$ in each K\"{a}hler class
in $H_c^2(Y,\R)\subset H^2(Y,\R)$ which is asymptotic to the K\"{a}hler cone metric $\ol{g}$ on $X$ as follows.
There is an $R>0$ such that, for any $\delta>0$ and $k\geq 0$,
\begin{equation}\label{eq:asymp}
\nabla^k \left(\pi_* g -\ol{g}\right) =O\left( r^{-2n+\delta-k}\right)\quad\text{on }\{y\in C(S):r(y)>R\},
\end{equation}
where $\nabla$ is the covariant derivative of $\ol{g}$.
\end{thm}

We also considered the toric case.  In this case $X=C(S)$
is a Gorenstein toric K\"{a}hler cone which admits a toric Ricci-flat K\"{a}hler cone metric by the
results of~\cite{FOW}.  In this case a crepant resolution $\pi:Y\rightarrow X$ is toric, and $Y$ is described
explicitly by a nonsingular simplicial fan $\tilde{\Delta}$ refining the convex polyhedral cone $\Delta$ defining $X$.
To ensure that there is a K\"{a}hler class in $H_c^2(Y,\R)$ we require that $\tilde{\Delta}$ defining $\pi:Y\rightarrow X$ admits a \emph{compact} strictly convex support function.  This is a strictly convex support function on $\tilde{\Delta}$
satisfying the additional condition that it vanishes on the rays defining $\Delta$.  We prove the following.
\begin{cor}[\cite{vC2}]\label{cor:main}
Let $\pi:Y\rightarrow X$ be a crepant resolution of a Gorenstein toric K\"{a}hler cone $X$ with an isolated
singularity.  Suppose the fan $\tilde{\Delta}$ defining $Y$ admits a compact strictly convex support function.
Then $Y$ admits a Ricci-flat K\"{a}hler metric $g$ which is asymptotic to $(C(S),\ol{g})$ as in
(\ref{eq:asymp}).  Furthermore, $g$ is invariant under the compact $n$-torus $T^n$.
\end{cor}

This article further considers the toric case, and we find more examples using Corollary~\ref{cor:main}.
Already in~\cite{vC2} many examples of crepant resolutions of Gorenstein toric K\"{a}hler cones which satisfy
Corollary~\ref{cor:main} are given.  But for $n=3$ a much more exhaustive existence result can be given.
This is fortuitous as one source of interest in these asymptotically conical Calabi-Yau manifolds is in the AdS/CFT
correspondence (cf.~\cite{MS4,MS3}) and the $n=3$ case is of primary importance.

We show that for $n=3$ any Gorenstein toric K\"{a}hler cone $X$ admits a crepant resolution $\pi:Y\rightarrow X$
such that the fan $\tilde{\Delta}$ defining $Y$ admits a compact strictly convex support function provided
$X$ is not a terminal singularity.  The only Gorenstein toric K\"{a}hler cone for $n=3$ with a terminal singularity
is the quadric hypersurface $X=\{z_0^2 +z_1^2 +z_2^2 +z_3^2 =0\}\subset\C^4$.  Therefore we get the following.
\begin{thm}\label{thm:main-tor}
Let $X$ be a three dimensional Gorenstein toric K\"{a}hler cone with an isolated singularity which is not
the quadric hypersurface, as a variety.  Then there is a crepant resolution $\pi:Y\rightarrow X$ such that $Y$
admits a Ricci-flat K\"{a}hler metric $g$ which is asymptotic to $(C(S),\ol{g})$ as in (\ref{eq:asymp}).
Furthermore, $g$ is invariant under the compact torus $T^3$.
\end{thm}
The proof is simple application of \emph{generalized} toric blow-ups of $X$ at points and along curves.
A similar argument shows that in dimensions $n\geq 4$ a Gorenstein toric K\"{a}hler cone $X$ admits a crepant
partial resolution $\pi:Y\rightarrow X$ such that $Y$ has only orbifold singularities and which satisfies
Corollary~\ref{cor:main}.  Note that the quadric hypersurface admits a small resolution
$\pi:Y\rightarrow X$ with exceptional set $\pi^{-1}(o)=\cps^1$.  And $Y$ admits a Ricci-flat K\"{a}hler
metric which is asymptotic to the cone over the homogeneous Sasaki-Einstein structure on $S^2 \times S^3$.
But the convergence in (\ref{eq:asymp}) is replaced by $O(r^{-2-k})$ (cf.~\cite{CanOs}).

Before we prove Theorem~\ref{thm:main-tor} we cover the toric geometry that is needed
to describe toric K\"{a}hler cones and their resolutions.  We also give some results on resolutions of
quotient singularities which will be used later.

Examples of Gorenstein toric K\"{a}hler cones, equivalently toric Sasaki-Einstein 5-manifolds, are known in
abundance (cf.~\cite{vC1,vC2} and~\cite{CFO}).  Note that if the Sasaki link $S\subset X$ is simply connected,
then in the toric case $H_2(S,\Z)=\Z^r$ and Smale's classification of 5-manifolds implies that
$S\overset{\text{diff}}{\cong} \# k(S^2 \times S^3)$.  Thus using Theorem~\ref{thm:main-tor} and the examples of
~\cite{vC1,vC2} and~\cite{CFO} we produce Ricci-flat asymptotically conical K\"{a}hler manifolds $Y$ asymptotic to
cones over $\# k(S^2 \times S^3)$.  And in fact, this produces infinitely many examples for each $k\geq 1$.

We also consider examples given by resolutions of hypersurface singularities.
There has been much research recently in constructing examples of Sasaki-Einstein manifolds, and many of the constructions
involve quasi-homogeneous hypersurface singularities.  The link $S\subset X$ of the singularity admits a Sasaki structure and
techniques have been developed to prove it admits a Sasaki-Einstein structure (cf.~\cite{BGN1,BGK,BG2}).
These Sasaki-Einstein manifolds provide many examples of Ricci-flat K\"{a}hler cones for which we can try to find a crepant resolution
and apply Theorem~\ref{thm:main}.
Again we concentrate on the $n=3$ case.  In this case much is known on the existence of crepant resolutions of these singularities.
We review much of this in Section~\ref{sec:resolutions}.  We also give some useful results on the algebraic geometry and topology of
crepant resolutions $\pi:Y\rightarrow X$ when they exist.  We then give some families of examples.
The first group are hypersurface singularities for which the terminalization procedure of M. Reid, first described in
~\cite{Re1}, can be carried out without much difficulty.  These examples have $b_3 (Y) \neq 0$ in contrast with the toric case.
And many of the Sasaki links $S\subset Y$ are rational homology spheres.  Some properties of the links $S$ and
the resolved spaces $Y$ are listed in Figure~\ref{fig:b3neq0}.

We then consider a family of examples which are also resolutions of quasi-homogeneous hypersurfaces.  In this case
they are resolved to orbifolds and then the quotient singularities are resolved.  This uses the fact that every
Gorenstein quotient singularity in dimension three admits a crepant resolution.  The links of the hypersurfaces in most of these examples
were proved to admit Sasaki-Einstein metrics in ~\cite{JoKol} and ~\cite{BGN1,BGN2}.
These examples all have $b_3(Y)=0$.
The topological types of the links $S\subset Y$ and properties of the resolution $Y$ are listed in Figure~\ref{fig:b3eq0}.

\section{K\"{a}hler cones and Sasaki manifolds}

\subsection{Introduction}

We review some of the properties of Sasaki manifolds.  We are primarily interested in K\"{a}hler cones, and in particular
Ricci-flat K\"{a}hler cones.  But a K\"{a}hler cone is a cone over a Sasaki manifold and is Ricci-flat precisely when the Sasaki
manifold is Einstein.  And there has been much research recently on Sasaki-Einstein manifolds (cf.~\cite{BG1,BG2,FOW}).
See~\cite{BG2} for more details on Sasaki manifolds.
\begin{defn}
 A Riemannian manifold $(S,g)$ of dimension $2n-1$ is Sasaki if the metric cone $(C(S),\ol{g})$, $C(S)=\R_{>0}\times S$
and $\ol{g}=dr^2 +r^2 g$, is K\"{a}hler.
\end{defn}

Set $\tilde{\xi}=J(r\frac{\partial}{\partial r})$, then $\tilde{\xi}-iJ\tilde{\xi}$ is a holomorphic vector field
on $C(S)$.  The restriction $\xi$ of $\tilde{\xi}$ to $S=\{r=1\}\subset C(S)$
is the \emph{Reeb vector field} of $S$, which is a Killing vector field.
If the orbits of $\xi$ close, then it defines a locally free $U(1)$-action on $S$.  If the $U(1)$-action is free,
then the Sasaki structure is said to be \emph{regular}.  If there are non-trivial stabilizers then the Sasaki
structure is \emph{quasi-regular}.  If the orbits do not close the Sasaki structure is \emph{irregular}.
The closure of the subgroup of the isometry group generated by $\xi$ is a torus $T^k$.  We define the
\emph{rank} of the Sasaki manifold to be $\rank (S):=k$.

Let $\eta$ be the dual 1-form to $\xi$ with respect to $g$.  Then
\begin{equation}\label{eq:cont}
\eta = (2d^c \log r)|_{r=1},
\end{equation}
where $d^c=\frac{i}{2}(\ol{\partial}-\partial)$.  Let $D=\ker\eta$.  Then $d\eta$ in non-degenerate on $D$
and $\eta$ is a contact form on $S$.  Furthermore, we have
\begin{equation}
d\eta(X,Y) =2g(\Phi X,Y), \quad\text{for }X,Y\in D_x,\ x\in\mathcal{S},
\end{equation}
where $\Phi$ defined by $\Phi(V) =JV$ for $V\in D_x$, and $\Phi(\xi)=0$. Thus
$(D,J)$ is a strictly pseudo-convex CR structure on $S$.
We will denote a Sasaki structure on $S$ by $(g,\xi,\eta,\Phi)$.
It follows from (\ref{eq:cont}) that the K\"{a}hler form of $(C(S),\ol{g})$ is
\begin{equation}\label{eq:kahler-form}
 \omega=\frac{1}{2}d(r^2 \eta)=\frac{1}{2}dd^c r^2.
\end{equation}
Thus $\frac{1}{2}r^2$ is a K\"{a}hler potential for $\omega$.

There is a 1-dimensional foliation $\mathscr{F}_\xi$ generated by the Reeb vector field $\xi$.  Since the leaf
space is identical with that generated by $\tilde{\xi}-iJ\tilde{\xi}$ on $C(S)$, $\mathscr{F}_\xi$ has
a natural transverse holomorphic structure.  And $\omega^T =\frac{1}{2}d\eta$ defines a K\"{a}hler form on the
leaf space.  We will denote the transverse K\"{a}hler metric by $g^T$.
Note that when the Sasaki structure on $S$ is regular (resp. quasi-regular), the leaf space of
$\mathscr{F}_\xi$ is a K\"{a}hler manifold (resp. orbifold).

A p-form $\alpha\in\Omega^p(S)$ on $S$ is said to be \emph{basic} if
\begin{equation}
\xi\contr\alpha =0\quad\text{and}\quad \mathcal{L}_\xi \alpha =0.
\end{equation}
The basic p-forms are denoted by $\Omega_B^p(S)$, where the foliation $\mathscr{F}_\xi$ on $S$ must be fixed.
One easily checks that $\Omega_B^*$ is closed under the exterior derivative.  So we have the basic de Rham complex
\begin{equation}
0\longrightarrow\Omega^0_B \overset{d}{\longrightarrow}\Omega^1_B \overset{d}{\longrightarrow}\cdots\longrightarrow\Omega^{2n-2}_B \longrightarrow 0,
\end{equation}
and the basic cohomology $H^*_B(S)$.

The foliation $\mathscr{F}_\xi$ associated to a Sasaki structure has a transverse holomorphic structure, so there is a
splitting $\Omega_B ^k =\bigoplus_{p+q=k} \Omega_B^{p,q}$ of complex forms into types.  And the exterior derivative on
basic forms splits into $d=\partial +\ol{\partial}$, where $\partial$ has degree $(1,0)$ and $\ol{\partial}$ has degree $(0,1)$.
Thus we have as well the basic Dolbeault complex
\begin{equation}
0\longrightarrow\Omega^{p,0}_B \overset{\ol{\partial}}{\longrightarrow}\Omega^{p,1}_B \overset{\ol{\partial}}{\longrightarrow}\cdots\longrightarrow\Omega^{p,2n-2}_B \longrightarrow 0,
\end{equation}
and the basic Dolbeault cohomology groups $H^{p,q}_B(S)$.

Furthermore, the foliation has a transverse K\"{a}hler structure, and the usual Hodge theory for K\"{a}hler manifolds carries over.
In particular, we have the Hodge decomposition $H^k_B(S,\C)=\bigoplus_{p+q=k} H^{p,q}_B(S)$ and the representation of basic cohomology classes by harmonic forms.  It is also useful to know that the $\partial\ol{\partial}$-lemma holds for basic forms as it does on K\"{a}hler
manifolds.  Thus if $\phi\in\Omega^{1,1}_B$ is exact, then there is a basic $f\in C_B^{\infty}$ with $\phi=i\partial\ol{\partial}f$ and
$f$ can be taken to be real if $\phi$ is.  See the monograph~\cite{BG2} for a survey of these results.

We will consider deformations of the transverse K\"{a}hler structure.  Let $\phi\in C^\infty_B(S)$ be a smooth basic
function.  Then set
\begin{equation}
\tilde{\eta} =\eta +2d^c_B \phi.
\end{equation}
Then
\[ d\tilde{\eta} =d\eta +2d_B d^c_B \phi =d\eta +2i\partial_B \ol{\partial}_B \phi. \]
For sufficiently small $\phi$, $\tilde{\eta}$ is a non-degenerate contact form in that $\tilde{\eta}\wedge d\tilde{\eta}^n$
is nowhere zero.  Then we have a new Sasaki structure on $S$ with the same Reeb vector field $\xi$, transverse
holomorphic structure on $\mathscr{F}_\xi$, and holomorphic structure on $C(S)$.  This Sasaki structure has
transverse K\"{a}hler form $\tilde{\omega}^T=\omega^T +i\partial_B \ol{\partial}_B \phi$.  One can show~\cite{FOW}
that if
\[\tilde{r} =r\exp{\phi},\]
then $\tilde{\omega}=\frac{1}{2}dd^c \tilde{r}^2$ is the K\"{a}hler form on $C(S)$ associated to the
transversally deformed Sasaki structure.

\begin{xpl}\label{ex:Sasak-st}
Let $Z$ be a complex manifold (or orbifold) with a negative holomorphic line bundle (respectively V-bundle) $\mathbf{L}$.
If the total space of $\mathbf{L}^\times$, $\mathbf{L}$ minus the zero section, is smooth, then the $U(1)$-subbundle
$S\subset\mathbf{L}^\times$ has a natural regular (respectively quasi-regular) Sasaki structure.
Let $h$ be an Hermitian metric on $\mathbf{L}$ with negative curvature.  If in local holomorphic coordinates we define
$r^2 =h|z|^2$, where $z$ is the fiber coordinate, then $\omega =\frac{1}{2}dd^c r^2$ is the K\"{a}hler form on
$\mathbf{L}^\times$ of a K\"{a}hler cone metric.  And $S=\{z\in\mathbf{L}^\times : r(z)=1\}$ has the induce Sasaki structure.

Conversely, it can be shown that every regular (respectively quasi-regular) Sasaki structure arises from this construction
(cf.~\cite{BG}).
\end{xpl}

\begin{xpl}\label{ex:Sasak-sph}
This example will construct, up to automorphism, all the Sasaki structures on the sphere $S^{2n+1}$ with the standard
CR-structure, i.e. the CR-structure induced by the usual embedding $S^{2n+1}\subset\C^{n+1}$.  We denote by
$z_j=x_j +iy_j, j=0,\ldots,n$, the coordinates on $\C^{n+1}=\R^{2n+2}$.  The standard CR-structure $(D,J)$ is given by the kernel of
$\eta =\sum_{j=0}^n x_j dy_j -y_j dx_j$, with $J$ induced by the embedding $S^{2n+1}\subset\C^{n+1}$.  Let $\mathbf{\lambda}=(\lambda_0,\ldots, \lambda_n)\in (\R_+)^{n+1}$.
Then we have the action induced by by the diagonal matrix $\mathbf{\lambda}$ with vector field $X_{\lambda}$ and
$JX_{\lambda}=xi_{\mathbf{\lambda}}$ given by
\begin{equation}
\xi_{\mathbf{\lambda}} =\sum_{j=0}^n \lambda_j(x_j \partial_{y_j} -y_j \partial_{x_j}).
\end{equation}
Then as in~\cite{BGS2} there is a unique Sasaki structure $(g,\xi_{\mathbf{\lambda}},\eta_{\lambda},\Phi)$, denoted by
$S^{2n+1}_{\mathbf{\lambda}}$, with Reeb vector field $\xi_{\mathbf{\lambda}}$ and with the CR-structure $(J,D)$.
The contact form $\eta_{\mathbf{\lambda}}$ is
\begin{equation}\label{eq:sasak-sph}
 \eta_{\mathbf\lambda} =\frac{\sum_{j=0}^n (x_j dy_j -y_j dx_j)}{\sum_{j=0}^{n}\lambda_j(x_j^2 +y_j^2)}.
\end{equation}
The K\"{a}hler cone $C(S^{2n+1}_{\mathbf{\lambda}})$ can be identified biholomorphically with $\C^{n+1}\setminus\{0\}$.
This can be seen using the action of the Euler vector field $-J\xi_{\mathbf\lambda}$.

Conversely, a Sasaki structure on $S^{2n+1}$ with the standard CR-structure is given by a vector field $\xi$ on $S^{2n+1}$ transversal
to $D$ and inducing an automorphism of $(D,J)$.  The group of CR-automorphism of $(D,J)$ is known to be $SU(n+1,1)$.
The action of $SU(n+1,1)$ extends to the ball $B\subset\C^{n+1}$ bounded by $S^{2n+1}$.  Identify $B$ with the positive cone
\[ \{v\in V\ :\ (v,v)>0 \}, \]
where $(\cdot,\cdot)$ is the Hermitian form on $V=\C^{n+2}$ with signature $(n+1,1)$.  Then clearly $SU(n+1,1)$ acts transitively on
the interior of $B\subset\mathbb{P}(V)$.  The vector field $\xi$ must be induced by an element of $\mathfrak{su}(n+1,1)$, which
we denote by $\xi$ again.  The flow generated by $\xi$ on $B$ must have a fixed point $x\in\Int(B)$.  By conjugating by an element
$g\in SU(n+1,1)$ we may assume that $\xi$ vanishes at $0\in\C^{n+1}$.  Then $\xi\in\mathfrak{u}(n+1)$.  And by conjugating by
an element of $U(n+1)\subset SU(n+1,1)$ we may assume that $\xi$ is represented by a diagonal matrix with eigenvalues
$\alpha_j ,\ j=0,\ldots,n$.  Since $\xi$ is transversal to $D$, the real numbers $\lambda_j =-\sqrt{-1}\alpha_j$ are positive;
and we may assume $0<\lambda_0<\cdots<\lambda_n$.  This is clearly the Sasaki structure on $S^{2n+1}_{\mathbf{\lambda}}$ constructed
above.
\end{xpl}

\begin{prop}\label{prop:ricci}
Let $(S,g)$ be a $2n-1$-dimensional Sasaki manifold.  Then the following are equivalent.
\begin{thmlist}
 \item $(S,g)$ is Sasaki-Einstein with the Einstein constant being necessarily $2n-2$.

 \item $(C(S),\ol{g})$ is a Ricci-flat K\"{a}hler.

 \item The K\"{a}hler structure on the leaf space of $\mathscr{F}_\xi$ is K\"{a}hler-Einstein with Einstein constant $2n$.
\end{thmlist}
\end{prop}
This follows from elementary computations.  In particular, the equivalence of (i) and (iii) follows from
\begin{equation}\label{eq:ricci}
 \Ric_g(\tilde{X},\tilde{Y})=(\Ric^T -2g^T)(X,Y),
\end{equation}
where $\tilde{X},\tilde{Y}\in D$ are lifts of $X,Y$ in the local leaf space.

Given a Sasaki structure we can perform a $D$-homothetic transformation to get a new Sasaki structure.  For $a>0$ set
\begin{gather}\label{eq:d-homo}
 \eta'=a\eta,\quad \xi'=\frac{1}{a}\xi,\\
 g'=ag^T +a^2\eta\otimes\eta =ag+(a^2-a)\eta\otimes\eta.
\end{gather}
Then $(g',\xi',\eta',\Phi)$ is a Sasaki structure with the same holomorphic structure on $C(S)$, and with
$r'=r^a$.

\begin{prop}\label{prop:CY-cond}
The following necessary conditions for $S$ to admit a deformation of the transverse K\"{a}hler structure to a Sasaki-Einstein
metric are equivalent.
\begin{thmlist}
 \item $c_1^B =a[d\eta]$ for some positive constant $a$.

 \item $c_1^B >0$, i.e. represented by a positive $(1,1)$-form, and $c_1(D)=0$.

 \item For some positive integer $\ell>0$, the $\ell$-th power of the canonical line bundle
 $\mathbf{K}^{\ell}_{C(S)}$ admits a nowhere vanishing section $\Omega$ with
 $\mathcal{L}_\xi \Omega =i n\Omega$.
\end{thmlist}
\end{prop}
\begin{proof}
 Let $\rho$ denote the Ricci form of $(C(S),\ol{g})$, then easy computation shows that
 \begin{equation}\label{eq:ricci-cone}
  \rho =\rho^T - 2n\frac{1}{2}d\eta.
 \end{equation}
If (i) is satisfied, there is a $D$-homothety so that $[\rho^T]= 2n[\frac{1}{2}d\eta]$ as basic classes.
Thus there exists a smooth function $h$ with $\xi h=0=r\frac{\partial}{\partial r}h$ and
\begin{equation}
 \rho =i\partial\ol{\partial}h.
\end{equation}
This implies that $e^h \frac{\omega^{n}}{n!}$, where $\omega$ is the K\"{a}hler form of $\ol{g}$, defines
a flat metric $|\cdot|$ on $\mathbf{K}_{C(S) }$.  Parallel translation defines a multi-valued section which defines
a holomorphic section $\Omega$ of $\mathbf{K}^{\ell}_{C(S)}$ for some integer $\ell>0$ with
$|\Omega |=1$.  Then we have
\begin{equation}\label{eq:hol-form}
 \left(\frac{i}{2}\right)^{n}(-1)^{\frac{n(n-1)}{2}}\Omega\wedge\ol{\Omega} =e^h\frac{1}{n!}\omega^{n}.
\end{equation}
  From the invariance of $h$ and the fact that $\omega$ is homogeneous of degree 2, we see that
  $\mathcal{L}_{r\frac{\partial}{\partial r}}\Omega= n\Omega$.

  The equivalence of (i) and (ii) is easy (cf.~\cite{FOW} Proposition 4.3).
\end{proof}

\subsection{Toric geometry}

In this section we recall the basics of toric Sasaki manifolds.  Much of what follows can be found in
~\cite{MSY} or~\cite{FOW}.

\begin{defn}
 A Sasaki manifold $(S,g)$ of dimension $2n-1$ is \emph{toric} if there is an effective action of an
 $n$-dimensional torus $T=T^{n}$ preserving the Sasaki structure such that the Reeb vector field $\xi$ is an
 element of the Lie algebra $\mathfrak{t}$ of $T$.

 Equivalently, a toric Sasaki manifold is a Sasaki manifold $S$ whose K\"{a}hler cone
 $C(S)$ is a toric K\"{a}hler manifold.
\end{defn}

We have an effective holomorphic action of $T_\C \cong (\C^*)^{n}$ on $C(S)$ whose restriction to
$T\subset T_\C$ preserves the K\"{a}hler form $\omega =d(\frac{1}{2}r^2 \eta)$.  So there is a moment map
\begin{equation}\label{eq:moment-map}
\begin{gathered}
 \mu: C(S) \longrightarrow \mathfrak{t}^* \\
 \langle \mu(x),X\rangle = \frac{1}{2}r^2\eta(X_S (x)),
\end{gathered}
\end{equation}
where $X_S$ denotes the vector field on $C(S)$ induced by $X\in\mathfrak{t}$.  We have the
moment cone defined by
\begin{equation}
 \mathcal{C}(\mu) :=\mu(C(S)) \cup \{0\},
\end{equation}
which from~\cite{Ler} is a strictly convex rational polyhedral cone.  Recall that this means that there are vectors
$u_i,i=1,\ldots,d$ in the integral lattice $\Z_T =\ker\{\exp(2\pi i\cdot):\mathfrak{t}\rightarrow T\}$ such that
\begin{equation}\label{eq:moment-cone}
 \mathcal{C}(\mu) =\bigcap_{j=1}^{d} \{y\in\mathfrak{t}^* : \langle u_j,y\rangle\geq 0\}.
\end{equation}
The condition that $\mathcal{C}(\mu)$ is strictly convex means that it is not contained in any linear subspace of $\mathfrak{t}^*$,
and it is cone over a finite polytope.
We assume that the set of vectors $\{u_j\}$ is minimal in that removing one changes the set defined by
(\ref{eq:moment-cone}).  And we furthermore assume that the vectors $u_j$ are primitive, meaning that
$u_j$ cannot be written as $p\tilde{u}_j$ for $p\in\Z,p>1,$ and $\tilde{u}_j\in\Z_T$.

Let $\Int\mathcal{C}(\mu)$ denote the interior of $\mathcal{C}(\mu)$.  Then the action of $T$ on
$\mu^{-1}(\Int\mathcal{C}(\mu))$ is free and is a Lagrangian torus fibration over $\Int\mathcal{C}(\mu)$.
There is a condition on the $\{u_j\}$ for $\mathcal{S}$ to be a smooth manifold.  Each face
$\mathcal{F}\subset\mathcal{C}(\mu)$ is the intersection of a number of facets
$\{y\in\mathfrak{t}^* :l_j(y)=\langle u_j ,y\rangle =0\}$.  Let $u_{j_1},\ldots,u_{j_a}$ be the corresponding
collection of normal vectors in $\{u_j\}$, where $a$ is the codimension of $\mathcal{F}$.  Then $S$ is smooth, and
the cone $\mathcal{C}(\mu)$ is said to be \emph{non-singular} if and only if
\begin{equation}\label{eq:nonsing}
 \left\{ \sum_{k=1}^{a} \nu_{k} u_{j_k} :\nu_k \in\R\right\}\cap\Z_T =\left\{\sum_{k=1}^{a} \nu_{k} u_{j_k} :\nu_k \in\Z\right\}
\end{equation}
for all faces $\mathcal{F}$.

Note that $\mu(S)=\{y\in\mathcal{C}(\mu) : y(\xi)=\frac{1}{2}\}$.  The hyperplane
$\{y\in\mathfrak{t}^* : y(\xi)=\frac{1}{2}\}$ is the \emph{characteristic hyperplane} of the Sasaki structure.
Consider the dual cone to $\mathcal{C}(\mu)$
\begin{equation}\label{eq:dual-cone}
 \mathcal{C}(\mu)^* =\{\tilde{x}\in\mathfrak{t} :\langle\tilde{x}, y\rangle\geq 0\text{ for all }y\in\mathcal{C}(\mu)\},
\end{equation}
which is also a strictly convex rational polyhedral cone by Farkas' theorem.  Then $\xi$ is in the interior of
$\mathcal{C}(\mu)^*$.  Let $\frac{\partial}{\partial\phi_i},i=1,\ldots, n$ be a basis of $\mathfrak{t}$ in $\Z_T$.
Then we have the identification $\mathfrak{t}^*\cong\mathfrak{t}\cong\R^{n}$ and write
\[ u_j =(u_j^1, \ldots,u_j^{n}),\quad \xi=(\xi^1,\ldots,\xi^{n}). \]
If we set
\begin{equation}\label{eq:symp-coord}
y_i =\langle\mu(x),\frac{\partial}{\partial\phi_i}\rangle\quad ,i=1,\ldots, n,
\end{equation}
then we have symplectic coordinates $(y,\phi)$ on $\mu^{-1}(\Int\mathcal{C}(\mu))\cong \Int\mathcal{C}(\mu)\times T^{n}$.
In these coordinates the symplectic form is
\begin{equation}
 \omega =\sum_{i=1}^{n} dy_i \wedge d\phi_i.
\end{equation}
The K\"{a}hler metric can be seen as in~\cite{Abr} to be of the form
\begin{equation}\label{eq:metric-sym}
 g=\sum_{ij} G_{ij}dy_i dy_j + G^{ij}d\phi_i d\phi_j,
\end{equation}
where $G^{ij}$ is the inverse matrix to $G_{ij}(y)$, and the complex structure is
\begin{equation}\label{eq:comp-st}
 \mathcal{I} =\left\lgroup
\begin{matrix}
 0 & -G^{ij} \\
 G_{ij} & 0 \\
\end{matrix}\right\rgroup
\end{equation}
in the coordinates $(y,\phi)$.  The integrability of $\mathcal{I}$ is $G_{ij,k} =G_{ik,j}$.  Thus
\begin{equation}
 G_{ij}=G_{,ij} :=\frac{\partial^2 G}{\partial y_i \partial y_j},
\end{equation}
for some strictly convex function $G(y)$ on $\Int\mathcal{C}(\mu)$.  We call $G$ the symplectic potential of the
K\"{a}hler metric.

One can construct a canonical K\"{a}hler structure on the cone $X=C(S)$, with a fixed holomorphic structure,
via a simple K\"{a}hler reduction of $\C^d$ (cf.~\cite{Gui1} and~\cite{BurGuiLer} for the singular case).
The symplectic potential of the canonical K\"{a}hler metric is
\begin{equation}
 G^{can} =\frac{1}{2}\sum_{i=1}^{d}l_i (y)\log l_i (y).
\end{equation}
Let
\[ G_\xi =\frac{1}{2}l_{\xi}(y)\log l_{\xi} -\frac{1}{2}l_{\infty}(y)\log l_{\infty}(y),\]
where
\[ l_{\xi}(y)=\langle\xi, y\rangle, \text{ and } l_{\infty}(y)=\sum_{i=1}^d \langle u_i ,y\rangle.\]
Then
\begin{equation}\label{eq:can-pot}
 G_\xi ^{can} =G^{can} + G_\xi,
\end{equation}
defines a symplectic potential of a K\"{a}hler metric on $C(S)$ with induced Reeb vector field $\xi$.
To see this write
\begin{equation}
 \xi =\sum_{i=1}^{n} \xi^i \frac{\partial}{\partial\phi_i},
\end{equation}
and note that the Euler vector field is
\begin{equation}
 r\frac{\partial}{\partial r} =2\sum_{i=1}^{n}y_i\frac{\partial}{\partial y_i}.
\end{equation}
Thus from (\ref{eq:comp-st}) we have
\begin{equation}\label{eq:sym-reeb}
 \xi^i =\sum_{j=1}^{n} 2G_{ij}y_j.
\end{equation}
Computing from (\ref{eq:can-pot}) gives
\begin{equation}\label{eq:can-pot-der}
 \left(G_\xi ^{can} \right)_{ij} =\frac{1}{2}\sum_{k=1}^d \frac{u_k^i u_k^j}{l_k (y)} +\frac{1}{2}\frac{\xi^i \xi^j}{l_\xi (y)}
 -\frac{1}{2}\frac{\sum_{k=1}^d u_k^i \sum_{k=1}^d u_k^j}{l_\infty (y)},
\end{equation}
which gives an explicit formula for the complex structure $\mathcal{I}$ in (\ref{eq:comp-st}) and the metric $g$ in (\ref{eq:metric-sym}).

The general symplectic potential is of the form
\begin{equation}
 G =G^{can} + G_\xi +g,
\end{equation}
where $g$ is a smooth homogeneous degree one function on $\mathcal{C}$ such that $G$ is strictly convex.
The following follows easily from this discussion.
\begin{prop}
 Let $S$ be a compact toric Sasaki manifold and $C(S)$ its K\"{a}hler cone.  For any
 $\xi\in\Int\mathcal{C}(\mu)^*$ there exists a toric K\"{a}hler cone metric, and associated Sasaki structure on $S$,
 with Reeb vector field $\xi$.  And any other such structure is a transverse K\"{a}hler deformation, i.e.
 $\tilde{\eta} =\eta +2d^c \phi$, for a basic function $\phi$.
\end{prop}

Consider now the holomorphic picture of $C(S)$.  Note that the complex structure on $X=C(S)$ is determined up to
biholomorphism by the associated moment polyhedral cone $\mathcal{C}(\mu)$ (cf.~\cite{Abr} Proposition A.1).
And the construction of $X=C(S)$ as in~\cite{Gui1,BurGuiLer} shows that $X=C(S)$ is a toric variety
with open dense orbit $(\C^*)^{n}\cong\mu^{-1}(\Int\mathcal{C})\subset C(S)$.

Recall that a toric variety is characterized by a \emph{fan} (cf.~\cite{Od}).  We give some definitions.
\begin{defn}
A subset $\sigma$ of $\mathfrak{t}\cong\R^n$ is a \emph{strongly convex rational polyhedral cone}, if
there exists a finite number of elements $u_1,u_2,\ldots,u_s$ in $\Z_T \cong\Z^n$
such that
\[\sigma =\{a_1 u_1 +\cdots+a_s u_s :a_i \in\R_{\geq 0}\text{ for }i=1,\ldots,s\},\]
and $\sigma\cap(-\sigma)=\{o\}$.
\end{defn}
\begin{defn}
 A \emph{fan} in $\Z_T \cong\Z^n$ is a nonempty collection $\Delta$  of strongly convex rational polyhedral cones
 in $\mathfrak{t}\cong\R^n$ satisfying the following:
\begin{thmlist}
\item  Every face of any $\sigma\in\Delta$ is contained in $\Delta$.

\item  For any $\sigma,\sigma'\in\Delta$, the intersection $\sigma\cap\sigma'$ is a face of both $\sigma$ and $\sigma'$.
\end{thmlist}
\end{defn}

We denote the dual cone to $\sigma$ by $\sigma^\vee \subset\mathfrak{t}^*$, and define
$\Z_T^\vee =\Hom(\Z_T,Z)$.
If $\sigma$ is a strongly convex rational polyhedral cone, then $U_\sigma =\Spec(\sigma^\vee \cap\Z_T^\vee)$
is an affine variety.  Given a fan $\Delta$ in $\Z_T \cong\Z^n$ the affine varieties $U_\sigma ,\sigma\in\Delta$
glue together to form a normal complex algebraic variety $X_{\Delta}$ with an algebraic action of
$T_{\C}\cong(\C^*)^n$.  Furthermore, there is an open dense orbit isomorphic to
$T_{\C}\cong(\C^*)^n$.  Conversely, if a torus $(\C^*)^n$ acts algebraically on a normal algebraic variety $X$, of
locally finite type over $\C$, with an open dense orbit isomorphic to $(\C^*)^n$, then there is a fan $\Delta$ in
$\Z^n$ with $X$ equivariantly isomorphic to $X_{\Delta}$.  See~\cite{Od} for more details.

There is a fan in $\Z_T \subset\mathfrak{t}$ associated to every strictly convex rational polyhedral set
$\mathcal{C}\subset\mathfrak{t}^*$.  Suppose
\begin{equation}
\mathcal{C}=\bigcap_{j=1}^{d} \{y\in\mathfrak{t}^* : \langle u_j,y\rangle\geq\lambda_j\},
\end{equation}
where $u_j \in\Z_T$ and $\lambda_j \in\R$ for $j=0,\ldots,d$.  Each face $\mathcal{F}\subset\mathcal{C}$ is the
intersection of facets $\{y\in\mathfrak{t}^* : l_{j_k}(y)=\langle u_{j_k},y\rangle-\lambda_{j_k} =0\}\cup\mathcal{C}$
for $k=1,\ldots, a$, where $\{j_1,\ldots,j_a\}\subseteq\{1,\ldots,d\}$, and the codimension of $\mathcal{F}$ is $a$.
Then to the face $\mathcal{F}$ we associate a cone $\sigma_{\mathcal{F}}$ in $\mathfrak{t}\cong\R^n$
\begin{equation}\label{eq:dual-cone-face}
\sigma_{\mathcal{F}} =\{c_1 u_{j_1} +\cdots+c_a u_{j_a} :c_{j_k} \in\R_{\geq 0}\text{ for }k=1,\ldots,a\}.
\end{equation}
It is easy to see that the set of all $\sigma_{\mathcal{F}}$ for faces $\mathcal{F}\subseteq\mathcal{C}$ define
a fan $\Delta$ in $\Z_T$.

Consider the convex polyhedral cone $\mathcal{C}(\mu)$.  From (\ref{eq:moment-cone}) the fan in $\Z_T$ associated to
$\mathcal{C}(\mu)$ consists of the dual cone (\ref{eq:dual-cone}) and all of its faces where
\begin{equation}\label{eq:dual-cone-fan}
\mathcal{C}(\mu)^* =\{c_{1} u_1 +\cdots +c_d u_d : c_k\in\R_{\geq 0}\text{ for }k=1,\ldots,d \}.
\end{equation}
It follows that $C(S)$ is an affine variety as its fan has a single n-dimensional cone.

We introduce logarithmic
coordinates $(z_1,\ldots,z_{n}) =(x_1 +i\phi_1,\ldots, x_{n}+i\phi_{n})$ on
$\C^{n}/{2\pi i\Z^{n}}\cong (\C^*)^{n}\cong\mu^{-1}(\Int\mathcal{C})\subset C(S)$, i.e.
$x_j +i\phi_j =\log w_j$ if $w_j,j=1,\ldots,n$, are the usual coordinates on $(\C^*)^{n}$.
The K\"{a}hler form can be written as
\begin{equation}
 \omega =\mathbf{i}\partial\ol{\partial}F,
\end{equation}
where $F$ is a strictly convex function of $(x_1,\ldots,x_{n})$.  One can check that
\begin{equation}
 F_{ij}(x) =G^{ij}(y),
\end{equation}
where $\mu=y=\frac{\partial F}{\partial x}$ is the moment map.  Furthermore, one can show $x=\frac{\partial G}{\partial y}$, and
the K\"{a}hler and symplectic potentials are related by the Legendre transform
\begin{equation}
 F(x)= \sum_{i=1}^{n} x_i \cdot y_i -G(y).
\end{equation}
It follows from equation (\ref{eq:symp-coord}) defining symplectic coordinates that
\begin{equation}
 F(x)= l_\xi (y) =\frac{r^2}{2}.
\end{equation}

We now consider the conditions in Proposition~\ref{prop:CY-cond} more closely in the toric case.
So suppose the Sasaki structure satisfies Proposition~\ref{prop:CY-cond}, thus we may assume
$c_1^B =2n[\omega^T]$.  Then equation (\ref{eq:ricci-cone}) implies that
\begin{equation}
 \rho =-i\partial\ol{\partial}\log\det(F_{ij})=i\partial\ol{\partial}h,
\end{equation}
with $\xi h=0=r\frac{\partial}{\partial r}h$, and we may assume $h$ is $T^{n}$-invariant.
Since a $T^{n}$-invariant pluriharmonic function is an affine function, we have
constants $\gamma_1,\ldots,\gamma_{n}\in\R$ so that
\begin{equation}\label{eq:hol-CY}
 \log\det(F_{ij})=-2\sum_{i=1}^{n}\gamma_i x_i -h.
\end{equation}
In symplectic coordinates we have
\begin{equation}\label{eq:symp-CY}
 \det(G_{ij})=\exp(2\sum_{i=1}^{n}\gamma_i G_i +h).
\end{equation}
Then from (\ref{eq:can-pot}) one computes the right hand side to get
\begin{equation}
 \det(G_{ij})=\prod_{k=1}^d \left(\frac{l_k (y)}{l_\infty (y)}\right)^{(\gamma,u_k)} (l_\xi (y))^{-n} \exp(h),
\end{equation}
And from (\ref{eq:can-pot-der}) we compute the left hand side of (\ref{eq:symp-CY})
\begin{equation}
 \det(G_{ij})=\prod_{k=1}^d(l_k (y))^{-1} f(y),
\end{equation}
where $f$ is a smooth function on $\mathcal{C}(\mu)$.  Thus $(\gamma, u_k)=-1$, for $k=1,\ldots,d$.
Since $\mathcal{C}(\mu)^*$ is strictly convex, $\gamma$ is a uniquely determined element of $\mathfrak{t}^*$.

Applying $\sum_{j=1}^{m}y_j \frac{\partial}{\partial y_j}$ to (\ref{eq:symp-CY}) and noting that
$\det(G_{ij})$ is homogeneous of degree $-n$ we get
\begin{equation}\label{eq:reeb-const}
 (\gamma,\xi)=-n.
\end{equation}

As in Proposition~\ref{prop:CY-cond} $e^h \det(F_{ij})$ defines a flat metric $\|\cdot\|$ on
$\mathbf{K}_{C(S)}$.  Consider the $(n,0)$-form
\[ \Omega = e^{i\theta}e^{\frac{h}{2}}\det(F_{ij})^{\frac{1}{2}} dz_1 \wedge\cdots\wedge dz_{n}.\]
From equation (\ref{eq:hol-CY}) we have
\[ \Omega =e^{i\theta}\exp(-\sum_{j=1}^{n}\gamma_j x_j)dz_1 \wedge\cdots\wedge dz_{n}.\]
If we set $\theta=-\sum_{j=1}^{n}\gamma_j \phi_j$, then
\begin{equation}\label{eq:hol-form-toric}
 \Omega=e^{-\sum_{j=1}^{n}\gamma_j z_j}dz_1 \wedge\cdots\wedge dz_{n}
\end{equation}
is clearly holomorphic on $U=\mu^{-1}(\Int\mathcal{C})$.  When $\gamma$ is not integral, then we take $\ell\in\Z_+$
such that $\ell\gamma$ is a primitive element of $\Z_T^* \cong\Z^{n}$.  Then
$\Omega^{\otimes\ell}$ is a holomorphic section of $\mathbf{K}^{\ell}_{C(S)}|_U$ which extends to
a holomorphic section of $\mathbf{K}^{\ell}_{C(S)}$ as $\|\Omega\|=1$.

It follows from (\ref{eq:hol-form-toric}) that
\begin{equation}
 \mathcal{L}_\xi \Omega =-i(\gamma,\xi)\Omega =i n\Omega.
\end{equation}
And note that we have equation (\ref{eq:hol-form}) from (\ref{eq:hol-CY}) and (\ref{eq:hol-form-toric}).
We collect these results in the following proposition.

\begin{prop}\label{prop:CY-cond-toric}
Let $\mathcal{S}$ be a compact toric Sasaki manifold of dimension $2n-1$.  Then the conditions of
Proposition~\ref{prop:CY-cond} are equivalent to the existence of $\gamma\in\mathfrak{t}^*$ such that
\begin{thmlist}
 \item $(\gamma, u_k)=-1$, for $k=1,\ldots,d$,

 \item $(\gamma,\xi)=-n$, and

 \item there exists $\ell\in\Z_+$ such that $\ell\gamma\in\Z_T^* \cong\Z^{n}$
\end{thmlist}
Then (\ref{eq:hol-form-toric}) defines a nowhere vanishing section of $\mathbf{K}^{\ell}_{C(S)}$.  And
$C(S)$ is $\ell$-Gorenstein if and only if a $\gamma$ satisfying the above exists.
\end{prop}

We will need the beautiful results of A. Futaki, H. Ono, and G. Wang on the existence of
Sasaki-Einstein metrics on toric Sasaki manifolds.
\begin{thm}[\cite{FOW,CFO}]\label{thm:FOW}
Suppose $S$ is a toric Sasaki manifold satisfying Proposition~\ref{prop:CY-cond-toric}.
Then we can deform the Sasaki structure by varying the Reeb vector field and then performing
a transverse K\"{a}hler deformation to a Sasaki-Einstein metric.  The Reeb vector field and transverse
K\"{a}hler deformation are unique up to isomorphism.
\end{thm}

In~\cite{FOW} a more general result is proved.  It is proved that a compact toric Sasaki
manifold satisfying Proposition~\ref{prop:CY-cond-toric} has a transverse K\"{a}hler deformation
to a Sasaki structure satisfying the transverse K\"{a}hler Ricci soliton equation:
\[ \rho^T - 2n\omega^T =\mathcal{L}_X \omega^T \]
for some Hamiltonian holomorphic vector field $X$.  The analogous result for toric Fano manifolds was
proved in~\cite{WaZh}.  A transverse K\"{a}hler Ricci soliton becomes a transverse K\"{a}hler-Einstein
metric, i.e. $X=0$, if the Futaki invariant $f_1$ of the transverse K\"{a}hler structure vanishes.
The invariant $f_1$ depends only on the Reeb vector field $\xi$.  The next step is to use a
volume minimization argument due to Martelli-Sparks-Yau~\cite{MSY} to show there is a unique $\xi$
satisfying (\ref{eq:reeb-const}) for which $f_1$ vanishes.

\begin{xpl}\label{xpl:two-points}
Let $M=\cps^2_{(2)}$ be the two-points blow up.  And Let $S \subset\mathbf{K}_M$ be the
$U(1)$-subbundle of the canonical bundle.  Then the standard Sasaki structure on $S$ as in Example~\ref{ex:Sasak-st}
satisfies (i) of Proposition~\ref{prop:CY-cond}, and it is not difficult to show that $S$
is simply connected and is toric.
But the automorphism group of $M$ is not reductive, thus $M$ does not admit a K\"{a}hler-Einstein metric
due to Y. Matsushima~\cite{Mat}.  Thus there is no Sasaki-Einstein structure with the usual Reeb vector field.
But by Theorem~\ref{thm:FOW} there is a Sasaki-Einstein structure with a different Reeb vector field.

The vectors defining the facets of $\mathcal{C}(\mu)$ are
\[u_1 =(0,0,1), u_2 =(0,1,1) ,u_3 =(1,2,1) ,u_4 =(2,1,1) ,u_5 =(1,0,1).\]
The Reeb vector field of the toric Sasaki-Einstein metric on $S$ was calculated in~\cite{MSY} to be
\[\xi =\left(\frac{9}{16}(-1+\sqrt{33}), \frac{9}{16}(-1+\sqrt{33}), 3\right).\]
One sees that the Sasaki structure is irregular with the closure of the generic orbit being a two torus.
\end{xpl}

\section{Resolutions}\label{sec:resolutions}

\subsection{Embeddings of K\"{a}hler cones}

In this section we will give a proof of the version of Kodaira-Nakano embedding appropriate to
K\"{a}hler cones and Sasaki manifolds.  There is a version due to W. Baily~\cite{Bai} giving an embedding of a
K\"{a}hler orbifold with a positive line V-bundle into projective space, $Z\hookrightarrow\cps^n$.
But even for a quasi-regular Sasaki manifolds $S$, this does not give a satisfying embedding.
The embedding $Z\hookrightarrow\cps^n$ does not respect the orbifold structure, so it
does not lift to an embedding $S\hookrightarrow S^{2n+1}$.  The appropriate embedding theorem for Sasaki
manifolds is due to L. Ornea and M. Verbitsky~\cite{OrVer}.  Their proof follows from a more general embedding
theorem for Viasman manifolds.  We give a proof more tailored to our context.

\begin{thm}\label{thm:embed}
Let $S$ be a compact Sasaki manifold with $H^1(S,\R)=0$.  There is a weighted Sasaki structure on the sphere $S_{\mathbf{w}}^{2n+1}$
for some $n>0$ and a CR-embedding $\iota: S\hookrightarrow S_{\mathbf{w}}^{2n+1}$.  The corresponding embedding of K\"{a}hler cones
$\phi: C(S)\cup\{o\}\hookrightarrow\C^{n+1}$ is holomorphic onto an affine subvariety of $\C^{n+1}$.

If $S$ is quasi-regular with associated K\"{a}hler orbifold $Z$, then we may assume that the Reeb vector field $\xi$ on $S$ is the
restriction of the Reeb vector field $\xi_{\mathbf{w}}$ on $S_{\mathbf{w}}^{2n+1}$.
And we have the commutative diagram
\begin{equation}\label{eq:Sasak-emb}
 \begin{CD}
  S @>\tau>> S^{2n+1}_{\mathbf{w}} \\
  @VV{\pi}V   @VVV\\
  Z @>\ol{\tau}>> \cps(\mathbf{w})
 \end{CD}
\end{equation}
where both rows are embeddings, $\ol{\iota}$ as complex orbifolds.  Furthermore, in this case by applying a transversal
K\"{a}hler deformation to $S^{2n+1}_{\mathbf{w}}$ the embeddings $\tau$ and $\ol{\iota}$ can be made to respect Sasaki and
K\"{a}hler structures respectively.
\end{thm}

If $H^1(S,\R)\neq 0$, then the proof still gives an holomorphic embedding $\phi: C(S)\cup\{o\}\hookrightarrow\C^{n+1}$.
But the part of the proof giving a CR-embedding may fail.

Note that the embedding $\tilde{\iota}:Z\hookrightarrow\cps(\mathbf{w})$ in (\ref{eq:Sasak-emb}) is of independent interest,
since it gives an orbifold embedding.  The singularities of $Z$ are all inherited from those of $\cps(\mathbf{w})$ as a subvariety.
The Baily embedding theorem merely gives an analytic embedding.

It is also interesting that Theorem~\ref{thm:embed} gives an alternative version of the Kodaira-Nakano embedding; although it
gives an embedding into a weighted projective space.
If $\mathbf{L}$ is a positive holomorphic bundle on a complex manifold $Z$, then the $U(1)$-subbundle of $\mathbf{L}^*$
has a regular Sasaki structure.  Thus (\ref{eq:Sasak-emb}) gives an embedding $\tilde{\iota}:Z\rightarrow\cps(\mathbf{w})$
whose image is disjoint from the orbifold singular set of $\cps(\mathbf{w})$.

\begin{proof}
Note that \textit{a priori} a K\"{a}hler cone $C(S)$ does not contain the vertex, but
$X=C(S)\cup\{o\}$ can be made into a complex space in a unique way.  The Reeb vector field $\xi$ generates
a 1-parameter subgroup of the automorphism group $\Aut(S)$ of the Sasaki manifold $S$.  Since $\Aut(S)$ is compact,
the closure of this subgroup is a torus $T^k \subset\Aut(S)$ where  $k=\rank(S)$.  Choose a vector field $\zeta$ in the
integral lattice of the Lie algebra $\mathfrak{t}$ of $T^k$, $\zeta\in\Z_T \subset\mathfrak{t}$, and such that $\eta(\zeta)>0$ on $S$.
Then as in~\cite{BGS1} on can show there is a unique Sasaki structure $(\tilde{g},\zeta,\tilde{\eta},\tilde{\Phi})$
with the same CR-structure and Reeb vector field $\zeta$.  The $U(1)$-action on $S$ generated by $\zeta$ extends to
an holomorphic $\C^*$-action on $C(S)$.  Then the quotient $C(S)/\C^* = S/ U(1)$ is a K\"{a}hler orbifold $Z$, and
$C(S)$ is the total space, minus the zero section, of an orbifold bundle $\pi:\mathbf{L}\rightarrow Z$
(cf.~\cite{BG}).  The bundle $\mathbf{L}$ is negative.  And there is an Hermitian metric $h$ on $\mathbf{L}$, so that
$\tilde{r}^2 =h|z|^2$, where $z$ is a local fiber coordinate, is the K\"{a}hler potential on $C(S)$ for the
Sasaki structure $(\tilde{g},\zeta,\tilde{\eta},\tilde{\Phi})$.

Let $W$ be the total space of $\mathbf{L}$.  Then $\tilde{r}^2$ is strictly
plurisubharmonic away from $Z\subset W$, and hence $W$ is a 1-convex space.  In other words, $W$ is exhausted by
strictly pseudo-convex domains $\{\tilde{r}^2 <c \}\subset W$ for $c>0$.  Then as in~\cite{Gra} $W$ is holomorphically
convex, and we have the Remmert reduction of $W$.  That is, there exists a Stein space $X$ and an holomorphic map
$\sigma: W\rightarrow X$, which contracts the maximal compact analytic set $Z\subset W$ and is a biholomorphism
outside $Z$.  Thus $X=C(S)\cup\{o\}$ is a normal complex space.  And the Riemann extension theorem
shows $i_* \mathcal{O}_{C(S)} =\mathcal{O}_X$, where $i: C(S)\rightarrow X$ is the inclusion.  Thus
$X$ is independent of the above choices.    And if $\pi:Y\rightarrow X$ is any
resolution of $o\in X$, then $Y$ is 1-convex.

The torus $T^k \subset\Aut(S), k=\rank(S),$ acts by holomorphic isometries on $X=C(S)\cup\{o\}$.  Let $U\subset X$ be a
$T^k$-invariant neighborhood of $o\in X$.  We can split $f\in\mathcal{O}(U)$ into its weight space components as follows.
Let $(a_1,\ldots,a_k)\in\Z^k$ and $t=(t_1,\ldots,t_k)\in T^k$.  Define
\begin{equation}
f_{(a_1,\ldots,a_k)}(z):=\frac{1}{(2\pi)^k}\int_{T^k} e^{-i\sum_{j=1}^k a_j t_j} f(t\cdot z)\, dt.
\end{equation}
We will show that
\begin{equation}\label{eq:weig-Laur}
f(z) =\sum_{(a_1,\ldots,a_k)\in\Z^k} f_{(a_1,\ldots,a_k)}(z),
\end{equation}
with the series on the right converging uniformly on compact subsets of $U\subset X$.
Let $x\in U\setminus\{o\}$, and let $B\subset\C^m$ be a ball mapped holomorphically $\beta:B\rightarrow U\setminus\{o\}$ with $\beta(0)=x$
transversal to the orbit of $(\C^*)^k$ through $x$.  Then $\psi:B\times(\C^*)^k \rightarrow X,\ \psi(z,t)=t\cdot\beta(z)$ maps
a neighborhood $D$ of $\{0\}\times T^k$ in $B\times(\C^*)^k$ onto a neighborhood of $x\in U\setminus\{o\}$.
And pulled back to $D$ (\ref{eq:weig-Laur}) is easily seen to be a Laurent expansion in k-coordinates which is known
to converge uniformly on compact subsets.  Thus (\ref{eq:weig-Laur}) converges uniformly on compact subsets of $U\setminus\{o\}$,
and it converges uniformly on compact subsets of $U$ by the maximum modulus theorem for analytic varieties~\cite[III B, Theorem 16]{GunRos}.

Now let $(f_1,\ldots,f_d): U\rightarrow\C^d$ be an embedding of a $T^k$-invariant neighborhood $o\in U$.  Here one can take $d$
to be the embedding dimension $\rank_o \Omega^1_X =\dim_{\C} m_o/m_o^2$, where $m_o$ is the maximal ideal at $o\in X$.
For the definition of the sheaf $\Omega^1_X$ for analytic spaces see~\cite[Ch. II]{GraPetRem}.
Then take sufficiently many components $(f_j)_{(a_1,\ldots,a_k)},\ j=1,\ldots,d, \ (a_1,\ldots,a_k)\in\Z^k$ to be the components of $\tau=(\tau_0,\ldots,\tau_n)$ so that $(d\tau)_o :\Omega^1_{\C^{n+1},0}\rightarrow\Omega^1_{U,o}$ is surjective.
As in the smooth case, an holomorphic map $\tau=(\tau_0,\ldots,\tau_n):U\rightarrow\C^{n+1}$ is an immersion at $o\in U$ if
$(d\tau)_o :\Omega^1_{\C^{n+1},0}\rightarrow\Omega^1_{U,o}$ is surjective.

For simplicity, denote the above quasi-regular Sasaki structure by $(g,\xi,\eta,\Phi)$.
The Reeb vector field $\xi$ of this Sasaki structure induces an holomorphic vector
field $-J\zeta -i\zeta$ whose action gives a 1-parameter subgroup $\gamma:\C^* \rightarrow T^k_{\C} =(\C^*)^k$ which can be
characterized by $b=(b_1,\ldots,b_k)\in\Z_T$.  By construction $\tau_j$ is in the weight space with weight $a(j)\in\Z^k$, so
$\gamma$ acts on $\tau_j$ with weight $w_j=\langle a(j),b\rangle$.  Since $\underset{t\rightarrow 0}{\lim}\gamma(t)\cdot x=o$
for $x\in X$, and $\tau_j (o)=0$, we have $w_j >0$.   And the weights $\mathbf{w}=(w_0,\ldots, w_n )\in(\Z_+)^{n+1}$ define a 1-parameter
subgroup $\gamma_{\mathbf{w}}:\C^* \rightarrow (\C^*)^{n+1}$ acting on $\C^{n+1}$.  Since $\gamma(t)$, for $t$ sufficiently close
to $0$, maps any compact subset of $X$ into $U$, each $\tau_j$ extends to $X$.  And a similar argument shows that
$\tau:X\rightarrow\C^{n+1}$ is an equivariant embedding.  Quotienting by $\gamma$ and $\gamma_{\mathbf{w}}$ gives $\ol{\tau}$ in
(\ref{eq:Sasak-emb}).

An argument similar to the proof of Chow's theorem can be used to prove that $V=\tau(X)\subset\C^{n+1}$ is an affine variety.
Let $f\in\mathcal{O}_{\C^{n+1},0}$ vanish on $V$ in a neighborhood of $0\in\C^{n+1}$.
Expand $f=\sum_{j=1}^{\infty}f_j$ into the weight components, each an homogeneous
polynomial with respect to $\gamma_{\mathbf{w}}$.  Then for a fixed $z$, $f(t\cdot z)=\sum_{j=1}^{\infty}f_j(z)t^j$ defines an holomorphic
function $t\mapsto f(t\cdot z)$ on $\C$.  If $z\in V$, then this function vanishes identically and each $f_j (z)=0$.
In a neighborhood $U$ of $0\in\C^{n+1}$ there are $f^1,\ldots, f^m \in\mathcal{O}_{\C^{n+1},0}$ so that
$V\cap U=\{z:f^1(z)=\cdots =f^m(z)=0\}$.  Then $V\cap U =\{z: f^i_j(z)=0,\ \forall 1\leq i\leq m, j\in\Z_+ \}$. But since
$\mathcal{O}_{\C^{n+1},0}$ is Noetherian, there are finitely many $\gamma_{\mathbf{w}}$-homogeneous polynomials
$f_1,\ldots,f_q$ (selected from the $f^i_j$) so that $V=\{z:f_1(z)=\cdots =f_q(z)=0\}$.

One sees that the the weighted Sasaki structure on $S^{2n+1}_{\mathbf{w}}$ restricts to a Sasaki structure on
$S' =V\cap S^{2n+1}_{\mathbf{w}}$ as the CR-structure on $S^{2n+1}_{\mathbf{w}}$ is compatible with the complex structure on $V$.
And the K\"{a}hler cone of $S'$ can be identified with $V$ by the action of the Euler vector field $-J\xi_{\mathbf{w}}$.

We have the holomorphic orbifold embedding $\ol{\tau}:Z\rightarrow Z' \subset\cps(\mathbf{w})$ which does not necessarily
preserve the K\"{a}hler structures.  As in~\cite{OrVer} we will use a result of J.-P. Demailly.
\begin{thm}[\cite{Dem}]\label{thm:Dem}
Let $(M,\omega)$ be a compact K\"{a}hler manifold, and $Z\subset M$ a closed complex submanifold. And let
$[\omega]\subset H^2(M)$ be the K\"{a}hler class of $\omega$. Consider a K\"{a}hler form $\omega_0$ on $Z$ such that its K\"{a}hler class
coincides with the restriction $[\omega]|_Z$. Then there exists a K\"{a}hler form $\omega'$on $M$ in the
same K\"{a}hler class as $\omega$, such that $\omega'|_Z =\omega_0$.
\end{thm}

The suborbifold $Z' \subset\cps(\mathbf{w})$ has a K\"{a}hler structure $\omega=\frac{1}{2}d\eta$ inherited from
the Sasaki structure $(g,\xi,\eta,\Phi)$ on $S$.  And $\cps(\mathbf{w})$ has the K\"{a}hler structure
$\omega_{\mathbf{w}} = \frac{1}{2}d\eta_{\mathbf{w}}$ inherited from the weighted Sasaki structure on $S^{2n+1}_{\mathbf{w}}$.
Since $\xi=\xi_{\mathbf{w}}$ on $S'$, $\alpha =\eta_{\mathbf{w}} -\eta$ is a basic form.
Thus $\omega_{\mathbf{w}} =\omega +\frac{1}{2}d\alpha$.  By Theorem~\ref{thm:Dem} there is an $f\in C^{\infty}(Z')$
so that $\omega'_{\mathbf{w}} =\omega_{\mathbf{w}} +dd^c f$ satisfies $\omega'_{\mathbf{w}}|_{Z'} =\omega$.

The transversally deformed Sasaki structure on $S^{2n+1}_{\mathbf{w}}$ has contact form
$\eta'_{\mathbf{w}}=\eta_{\mathbf{w}} +d^c f$.  And $\eta -\eta'_{\mathbf{w}}|_{S'}$ is closed.
Since $H^1(S,\R)=0$ there is an $h\in C_B^{\infty}(S')$ with $dh=\eta -\eta'_{\mathbf{w}}|_{S'}$.
Extend $h$ to a smooth basic function $h\in C_B^{\infty}(S^{2n+1}_{\mathbf{w}})$.  Then the contact form
$\tilde{\eta}_{\mathbf{w}}=\eta_\mathbf{w} +d^c f +dh$ on $S^{2n+1}_{\mathbf{w}}$ restricts to $\eta$ on $S'$.
And with the transversally deformed Sasaki structure $(\tilde{g},\xi_{\mathbf{w}},\tilde{\eta}_{\mathbf{w}},\tilde{\Phi}_{\mathbf{w}})$
on $S^{2n+1}_{\mathbf{w}}$, the embedding $\tau:S\rightarrow S^{2n+1}_{\mathbf{w}}$ preserves Sasaki structures.
\end{proof}

\subsection{Resolutions}

Let $\omega_X$ denote the dualizing sheaf of $X$.  Then we have $\omega_X \cong i_*(\mathcal{O}(\mathbf{K}_{C(S)}))$,
where $i:C(S)\rightarrow X$ is the inclusion, as the codimension of $\Sing(X)=\{o\}\subset X$ is greater than 1.
Recall that $X$ is said to be p-Gorenstein if $\omega_X^{[p]}:=i_*(\omega_{C(S)}^{\otimes p})$ is locally free for
$p\in\N$, and $X$ is $\Q$-Gorenstein if it is p-Gorenstein for some p.  We will call $X$ Gorenstein if it is
1-Gorenstein.

Recall that $X$ has rational singularities if $R^i \pi_* \mathcal{O}_Y =0$, for $i>0$, where
$\pi:Y\rightarrow X$ is a resolution of singularities.  If this holds for a resolution, then it holds for every
resolution.  If $o\in X$ is an isolated singularity, we have a simple criterion for
rationality (cf.~\cite{Bur} and~\cite{Lau}).

\begin{prop}\label{prop:ration}
Let $\Omega$ be a holomorphic $n$-form defined, and nowhere vanishing, on a deleted neighborhood of $o\in X$.
Then $o\in X$ is rational if and only if
\begin{equation}\label{eq:ration}
\int_U \Omega\wedge\ol{\Omega}<\infty,
\end{equation}
for $U$ a small neighborhood of $o\in X$.
\end{prop}
This implies that $\pi_*\omega_Y =\omega_X$ for any resolution $\pi:Y\rightarrow X$.  This has the following meaning.
\begin{defn}[\cite{Re1,Re2}]
We say that $X$ has \emph{canonical singularities} if it is normal and
\begin{thmlist}
 \item  For some $p\in\N$, $\omega_X^{[p]}$ is locally free,
 \item  $\pi_*\omega_Y^{\otimes p} =\omega_X^{[p]}$, for any resolution $\pi:Y\rightarrow X$.\label{def:item2}
\end{thmlist}
\end{defn}
Putting (ii) in terms of Weil divisors, we have
\begin{equation}\label{eq:discrep}
K_Y \equiv \pi^* K_X +\sum_i a(E_i,X)E_i,
\end{equation}
where $E_i$ are the $\pi$-exceptional divisors and $a(E_i,X)$ is the discrepancy.  Then (ii)
is equivalent to $a(E_i,X)\geq 0$ for all $E_i$.  Note that this condition is independent of the resolution.
In general canonical singularities are rational, but not all rational singularities are canonical.

A resolution $\pi:Y\rightarrow X$ is said to be \emph{crepant} if the discrepancy is zero.
In other words
\begin{equation}
\pi^* \omega_X =\omega_Y =\mathcal{O}(\mathbf{K}_Y).
\end{equation}
We are interested in finding examples of Ricci-flat K\"{a}hler cones $X=C(S)$ which admit a crepant resolution.

\begin{prop}\label{prop:resol}
Let $X=C(S)$ be the K\"{a}hler cone of a Sasaki manifold $S$ satisfying Proposition~\ref{prop:CY-cond},
e.g. $S$ is Sasaki-Einstein.  Then $X$ is $\Q$-Gorenstein, and $o\in X$ is a rational singularity.
If $X$ is Gorenstein, then $o\in X$ is a canonical singularity.

Suppose $X$ admits a crepant resolution $\pi:Y\rightarrow X$.
If $H_1(Y,\Z)=0$, which is always the case in dimension 3, then $X$ is Gorenstein.
\end{prop}
\begin{proof}
There exits a section $\Omega_p \in\Gamma(\mathbf{K}_{C(S)}^{\otimes p})$.  The Riemann extension theorem shows
that $\omega_X^{[p]} =i_*(\mathcal{O}(\mathbf{K}_{C(S)}^{\otimes p}))$ is locally free, and in fact trivial.
If $X$ is Gorenstein, then the holomorphic form $\Omega$ in Proposition~\ref{prop:CY-cond} is easily seen to
satisfy (\ref{eq:ration}).  In fact, the proof of Proposition~\ref{prop:ration} shows that $\Omega$ extends
to a regular form on any resolution of $X$.  Thus $o\in X$ is a canonical singularity.

Note that the conditions of Proposition~\ref{prop:CY-cond} imply that $\pi_1(S)$ is finite.  Indeed, the
transversal Ricci form $\ric^T \in [a\omega^T] ,\ a>0,$ where $a\omega^T$ is a positive basic $(1,1)$ class.
By the transverse version of the Calabi-Yau theorem there is a transversal K\"{a}hler deformation to a Sasaki
structure with $\Ric^T >0$.  Then after a possible $D$-homothetic transformation,  equation (\ref{eq:ricci})
shows that one can obtain a Sasaki metric with $\Ric_g >0$.  Then the claim follows by Meyer's theorem.

The universal cover $\ol{S}$ of $S$ is finite, and we have a
finite unramified morphism $g :\ol{X}\rightarrow X$, where $\ol{X}=C(\ol{S})\cup\{o\}$.  Then $\ol{X}$ has canonical
singularities by Proposition~\ref{prop:ration}.  It is well known that the image of a finite morphism $X$ must have
rational singularities~\cite[Prop. 5.13]{KolMor}.

By assumption $\pi^*\Omega_r$ is a nonvanishing section of
$\mathbf{K}_Y^{\otimes r}$.  Proposition~\ref{prop:pic} below proves that $\Pic Y =H^2(Y,\Z)$ which is free
by assumption.  Thus $\mathbf{K}_Y$ is trivial and has a nowhere vanishing section $\Omega$, and its restriction to
$i_*(\mathcal{O}(\mathbf{K}_{C(S)}))$ defines a nonvanishing section of $\omega_X$.

It is a result of N. Shepherd-Barron~\cite{Sh-B} than a crepant resolution of an isolated canonical 3-fold singularity is in fact
simply connected.
\end{proof}

Note that for any resolution $\pi:Y\rightarrow X$ the long exact sequence of the pair $(Y,S)$ (topologically $Y$ can be considered to have boundary $S$) and some arguments as in Theorem~\ref{thm:Betti} one can show that
$H_1(S,\Z)\rightarrow H_1(Y,\Z)$ is a surjection in any dimension.

We collect some properties of crepant resolutions $\pi:Y\rightarrow X$ of an isolated singularity that will be useful in the
sequel.  We are interested in the case in which $X=C(S)$ is the cone over a Sasaki manifold satisfying
Proposition~\ref{prop:CY-cond}, but the following results are true more generally for an isolated canonical singularity
$o\in X$ with $X$ Stein.

Recall that the \emph{local divisor class group} is
\begin{equation}
 \Cl(X,o):= \underset{\rightarrow}{\lim}\frac{\WDiv U}{\CDiv U},
\end{equation}
where the limit is over all open neighborhoods $U$ of $o$.  Since $X$ is contractible, $\Cl(X,o)=\Pic(X\setminus o)$.
Here we are considering analytic Weil divisors $\WDiv$ and analytic Cartier divisors $\CDiv$.
Thus $X$ is analytically factorial precisely when $\Cl(X,o)=0$.  Since $o\in X$ is an isolated rational singularity
it follows from Flenner~\cite[Satz (6.1)]{Fle} that
\begin{equation}\label{eq:Flen}
 \Cl(X,o) =H^2(S,\Z).
\end{equation}
Thus if $\pi_1(S)=e$, then $X$ is analytically factorial if it is analytically $\Q$-factorial.

In the case $\dim X=3$ much is known about the structure of canonical singularities.  We summarize some important results.
\begin{thm}\label{thm:can-3fold}
Let $X$ be a 3-dimesional analytic variety with only canonical singularities.  Then we have the following.
\begin{thmlist}
\item \emph{Terminalization~\cite[Main Theorem]{Re2}:} There is a projective crepant partial resolution $\pi:Y\rightarrow X$
where $Y$ has only terminal singularities.

\item \emph{$\Q$-factorialization~\cite[Corollary 5.4]{Kaw}:}  By possible further projective small resolutions $Y$ can be chosen
	to be analytically $\Q$-factorial.

\item  \emph{Quasi-uniqueness~\cite[Corollary 4.11]{Kol1}:}  $Y$ is generally not unique, but any two admit the same finite set of
	germs of isolated singularities.
\end{thmlist}
\end{thm}
The proof of Theorem~\ref{thm:can-3fold} (iii) involves showing that the two partial resolutions
differ by a finite sequence of birational modifications called flops.  These preserve the kind of singularities and
also $H^*(Y,\Z)$~\cite{Kol2}.  Thus in case $X=C(S)$ and $\dim X=3$, for a crepant resolution
$\pi:Y\rightarrow X$ $H^*(Y,\Z)$ is an invariant of the singularity $X=C(S)$.

We define the \emph{divisor class number} to be $\rho(X)=\rank\Cl(X,o)$.  A prime $\pi$-exceptional divisor $E$ is called
\emph{crepant} if its coefficient $a(E, X)=0$ in (\ref{eq:discrep}).  The number of crepant divisors, denoted $c(X)$, is finite and
independent of the resolution $\pi:Y\rightarrow X$.
\begin{prop}[\cite{Cai}]\label{prop:pic}
Let $o\in X$ be a canonical singularity with $X$ Stein, and let $\pi:Y\rightarrow X$ be a partial crepant resolution
of $X$.  Then $\Pic Y = H^2(Y,\Z)$.
\end{prop}
The singularities of $Y$ are canonical and thus rational.  An easy application of the Leray spectral sequence shows that
$H^i(Y,\mathcal{O}_Y)=0$ for $i>0$.  Then the proposition follows from the exponential sequence
\begin{equation}
 \rightarrow H^1(Y,\mathcal{O}_Y)\rightarrow\Pic Y \rightarrow H^2(Y,\Z)\rightarrow H^2(Y,\mathcal{O}_Y)\rightarrow.
\end{equation}

\begin{thm}[\cite{Cai}]\label{thm:Betti}
Suppose $o\in X$ is a canonical singularity as above.  Let $\pi:Y\rightarrow X$ be a
a partial crepant resolution with terminal analytically $\Q$-factorial singularities.  Then we have
\begin{thmlist}
 \item  $b_1(Y)=b_{2n-1}(Y) =b_{2n}(Y)=0$,

 \item  $b_{2n-2}(Y)= c(X)$,

 \item  $b_2(Y)=\rho(X) +c(X)$.
\end{thmlist}
\end{thm}
\begin{proof}
Arguments using the rationality of the singularities and the Leray spectral sequence similar to those in
Proposition~\ref{prop:pic} show $H^1(Y,\Z)=0$.  The retraction of $X$ to $o\in X$ lifts to a retraction of
$Y$ to $E=\pi^{-1}(o)$.  This proves (i).

For (ii) let $E_i$ be the prime divisors in $E=\pi^{-1}(o)$, then we have
\begin{equation}
H^{2n-2}(Y,\Z) \cong H^{2n-2}(E,\Z)\cong\bigoplus_{i=1}^{c(X)}H^{2n-2}(E_i ,\Z).
\end{equation}

The proof of (iii) follows from the following commutative diagram with exact rows
\begin{equation}
 \begin{CD}
  0 @>>> K @>>> \WDiv Y @>>> \WDiv X @>>> 0\\
  @.    @V=VV   @VVV         @VVV         @.\\
  0 @>>> K @>\iota>> \Cl(Y) @>>> \Cl(X) @>>> 0\\
 \end{CD}
\end{equation}
where $K$ is the group generated by divisors with support in $E=\pi^{-1}(o)$.
It is easy to see that $\iota$ is an inclusion because the codimension of $o\in X$ is at least 2.
Now (iii) follows from Proposition~\ref{prop:pic} and the fact that $Y$ is analytically $\Q$-factorial.
\end{proof}

Note that when $\dim X =3$ all of the Betti numbers besides $b_3(Y)$ are determined by invariants of $X$.
Since we are considering cone singularities $X=C(S)$, we will make use of (iii) most often in the form given by
(\ref{eq:Flen}).
\begin{equation}\label{eq:Betti}
b_2(Y) = b_2(S) + b_{2n-2}(Y).
\end{equation}

\subsection{Toric resolutions}\label{subsec:tor-res}

Let $X=C(S)$ be a toric K\"{a}hler cone.  Then as an algebraic variety $X=X_{\Delta}$ where $\Delta$ is the fan in
$\Z_T \cong\Z^n$ defined by the dual cone $\mathcal{C}(\mu)^*$, spanned by $u_1,\ldots,u_d\in\Z_T$, and its faces as
in (\ref{eq:dual-cone-fan}).  We assume that $X$ is Gorenstein.  Thus there is a $\gamma\in\Z_T^*$ so that
$\gamma(u_i)=-1$ for $i=1,\ldots, d$.  Let $H_\gamma =\{x\in\mathfrak{t}: \langle\gamma,x\rangle =-1\}$ be the hyperplane
defined by $\gamma$.  Then
\begin{equation}
P_{\Delta}:=\{x\in \mathcal{C}(\mu)^* : \langle\gamma,x\rangle =-1\}\subset H_\gamma \cong\R^{n-1}
\end{equation}
is an $(n-1)$-dimensional lattice polytope.  The lattice being $H_\gamma \cap\Z_T \cong\Z^{n-1}$.

A toric crepant resolution
\begin{equation}\label{eq:toric-resol}
\pi: X_{\tilde{\Delta}}\rightarrow X_{\Delta}
\end{equation}
is given by a nonsingular subdivision
$\tilde{\Delta}$ of $\Delta$ with every 1-dimensional cone $\tau_i \in\tilde{\Delta}(1), i=1,\ldots,N$ generated
by a primitive vector $u_i :=\tau_i \cap H_\gamma$.  This is equivalent to a basic, lattice triangulation of
$P_{\Delta}$.  \emph{Lattice} means that the vertices of every simplex are lattice points, and \emph{basic}
means that the vertices of every top dimensional simplex generates a basis of $\Z^{n-1}$.
Note that a \emph{maximal} triangulation of $P_{\Delta}$, meaning that the vertices of every simplex are its only
lattice points, always exists.  Every basic lattice triangulation is maximal, but the converse only holds in
dimension 2.

The condition that $u_i :=\tau_i \cap H_\gamma$ is primitive for each $i=1,\ldots,N$ is precisely the condition
that the section of Proposition (\ref{prop:CY-cond-toric}), $\Omega\in\Gamma(\mathbf{K}_{C(S)})$, characterized
by $\gamma\in\Z_T$ lifts to a non-vanishing section of $\mathbf{K}_{X_{\tilde{\Delta}}}$.
See~\cite{Od}, Proposition 2.1.

Note that a toric crepant resolution (\ref{eq:toric-resol}) of $X_{\Delta}$ is not unique, if one exists.
But if $E=\pi^{-1}(o)$ is the exceptional set, then the number of prime divisors in $E$ is invariant.
There is a prime divisor $E_i, i=d+1,\ldots,N$ for each lattice point $\Z_T \cap\Int P_{\Delta}$.

We give the toric version of Theorem~\ref{thm:Betti}.
\begin{thm}\label{thm:toric-Betti}
Let $X_{\Delta}$ be a n-dimensional Gorenstein toric cone with an isolated singularity $o\in X_{\Delta}$.
Let $\pi:X_{\tilde{\Delta}}\rightarrow X_{\Delta}$ be a
a partial crepant resolution with terminal analytically $\Q$-factorial singularities.  Then we have
\begin{thmlist}
 \item  $b_1(Y)=b_{2n-1}(Y) =b_{2n}(Y)=0$,

 \item  $b_{2n-2}(Y)= c(X)=|\Z_T \cap\Int P_{\Delta}|$,

 \item  $b_2(Y)=\rho(X) +c(X)=d-3 +|\Z_T \cap\Int P_{\Delta}|$.
\end{thmlist}
Furthermore, if $\tilde{\Delta}$ is simplicial, then $b_{2n-3} =0$.
\end{thm}
The proof follows from~\ref{thm:Betti} and the above remarks.  The last equality in (iii) follows from
the fact that if $S$ is the Sasaki link, $X_{\Delta}=C(S)$, then $b_2(S) =d-3$.
The final statement follows from the fact that each exceptional divisor $E_i$ has $b_{2n-3}(E_i) =0$ and an
easy argument with the Mayer-Vietoris sequence.

We are interested in resolutions $X_{\tilde{\Delta}}$ with K\"{a}hler classes, and in particular,
a K\"{a}hler classes in $H^2_c(X_{\tilde{\Delta}},\R)$.  We make some definitions to that end.
\begin{defn}
A real valued function $h:|\Delta|\rightarrow\R$ on the support $|\Delta|:=\cup_{\sigma\in\Delta}\sigma$ is a
\emph{support function} if it is linear on each $\sigma\in\Delta$.  That is, there exist an
$l_{\sigma}\in(\R^n)^*$ for each $\sigma\in\Delta$ so that $h(x)=\langle l_{\sigma},x\rangle$ for $x\in\sigma$,
and $\langle l_\sigma, x\rangle=\langle l_\tau ,x\rangle$ whenever $x\in\tau<\sigma$.
We denote by $\SF(\Delta,\R)$ the additive group of support functions on $\Delta$.
\end{defn}
We will always assume that $|\Delta|$ is a convex cone.
A support function $h\in\SF(\Delta,\R)$ is said to be \emph{convex} if $h(x+y)\geq h(x)+h(y)$ for any
$x,y\in|\Delta|$.  We have for $\sigma\in\Delta(n)$, $\langle l_\sigma, x\rangle \geq h(x)$ for all $x\in|\Delta|$.
If for every $\sigma\in\Delta(n)$, we have equality only for $x\in\sigma$, then $h$ is said to be
\emph{strictly convex}.

Given a strictly convex support function $h\in\SF(\tilde{\Delta})$,  we will associate a rational convex polyhedral
set $\mathcal{C}_h \subset\mathfrak{t}^*$ to $\tilde{\Delta}$ and $h$.  The fan associated to
$\mathcal{C}_h$ as in (\ref{eq:dual-cone-face}) is $\tilde{\Delta}$.
For each $\tau_j \in\tilde{\Delta}(1)$ we have a primitive element $u_j\in\Z_T,j=1,\ldots,N$ as above.
Set $\lambda_i :=h(u_i)$.  Then define
\begin{equation}\label{eq:polyt-def}
\mathcal{C}_h :=\bigcap_{j=1}^{N} \{y\in\mathfrak{t}^* : \langle u_j,y\rangle\geq\lambda_j\}.
\end{equation}
\begin{defn}
A strictly convex support function $h\in\SF(\tilde{\Delta},\R)$ is \emph{compact} if $h(u_j)=0, j=1,\ldots,d$,
where $u_j \in\Z_T ,j=1,\ldots,d,$ are the elements spanning $\Delta$.
\end{defn}

We will make use of a Hamiltonian reduction method of constructing a toric variety associated to a given
polyhedral set $\mathcal{C}_h \subset\mathfrak{t}^*$.  Originally due to Delzant and extended to the
non-compact and singular cases by D. Burns, V. Guillemin, and E. Lerman in~\cite{BurGuiLer} it constructs a
K\"{a}hler structure on $X_{\tilde{\Delta}}$ associated to a convex polyhedral set (\ref{eq:polyt-def}).  See also
\cite{Gui1,Gui2} for more on what is summarized here.

Let $\mathcal{A}:\Z^N \rightarrow\Z_T$ be the $\Z$-linear map with $\mathcal{A}(e_i)=u_i$, where $e_i,i=1,\ldots,N$
is the standard basis of $\Z^N$.  The $\R$-linear extension, also denoted by $\mathcal{A}$,  induces a map
of Lie algebras $\mathcal{A}:\R^N \rightarrow\mathfrak{t}$.  Let $\mathfrak{k}=\ker\mathcal{A}$.  We have an exact
sequence
\begin{equation}\label{eq:lie-alg}
0\rightarrow\mathfrak{k}\overset{\mathcal{B}}{\longrightarrow}\R^N \overset{\mathcal{A}}{\longrightarrow}\mathfrak{t}\rightarrow 0.
\end{equation}
Since $\mathcal{A}$ induces a surjective map of Lie groups if $T^N =\R^{N}/{2\pi\Z^{N}}$ and
$K=\ker\ol{\mathcal{A}}$, we have the exact sequence
\begin{equation}\label{eq:lie-group}
1\rightarrow K\longrightarrow T^N \overset{\ol{\mathcal{A}}}{\longrightarrow} T^n \rightarrow 1.
\end{equation}

The moment map $\Phi$ for the action of $T^N$ on $(\C^N ,\frac{\mathbf{i}}{2}\sum_{j=1}^N dz_j \wedge d\ol{z}_j )$ is
\begin{equation}\label{eq:moment}
\Phi(z) =\sum_{j=1}^N |z_j |^2 e_j^*.
\end{equation}
Then moment map $\Phi_K$ for the action of $K$ on $\C^N$ is the composition
\begin{equation}
\Phi_K =\mathcal{B}^* \circ\Phi,
\end{equation}
where $\mathcal{B}^*:(\R^N)^* \rightarrow\mathfrak{k}^*$ is the adjoint.
Let $\lambda =\sum_{j=1}^N \lambda_j e_j^*$, and $\nu =\mathcal{B}^* (-\lambda)$.  Then
\begin{equation}
M_{\mathcal{C}_h} :=\Phi_K^{-1}(\nu)/K
\end{equation}
is smooth provided $\mathcal{C}_h$ in non-singular as in (\ref{eq:nonsing}).  The K\"{a}hler form on $\C^N$ descends
to a K\"{a}hler form $\omega_h$ on $M_{\mathcal{C}_h}$.

\begin{prop}[\cite{vC3}]
We have $X_{\tilde{\Delta}}\cong M_{\mathcal{C}_h}$ as toric varieties.  Thus $\omega_h$ is a K\"{a}hler
form on $X_{\tilde{\Delta}}$ for any  strictly convex $h\in\SF(\tilde{\Delta},\R)$.

Furthermore, if $h$ is compact, then $[\omega_h ]\in H^2_c (X_{\tilde{\Delta}},\R)$.
The $u_j \in\Int P_{\Delta}, j=d+1,\ldots,N$, correspond to the prime divisors $D_j$ in $E=\pi^{-1}(o)$.
For each $j=d+1,\ldots,N$, let $c_j \in H^2_c (X_{\tilde{\Delta}},\R)$ be
the Poincar\'{e} dual of $[D_j]$ in $H_{2n-2}(X_{\tilde{\Delta}},\R)$.  Then
\[ [\omega]=[\omega_h] =-2\pi\sum_{j=d+1}^N \lambda_j c_j, \]
where $\omega$ is the Ricci-flat K\"{a}hler form of Corollary~\ref{cor:main}.
\end{prop}

\begin{prop}\label{prop:toric-res}
 Let $X=X_{\Delta}$ be a 3-dimensional Gorenstein toric cone variety.  Suppose $\Int{P}_{\Delta}$ contains a lattice points,
 i.e. $X$ is not a terminal singularity.  Then there is a basic lattice triangulation of $P_{\Delta}$ such that the corresponding
 subdivision $\tilde{\Delta}$ admits a compact strictly upper convex support function $h\in\SF(\tilde{\Delta},\R)$.
\end{prop}
\begin{proof}
 The subdivision of $\Delta$ is attained by a sequence of generalized toric blow-ups.  There is an integral $u\in\Int{P}_{\Delta}$
by hypothesis.  We let $\Delta_1$ be the star subdivision of $\Delta$ with respect to $\tau_u$, the 1-cone spanned by the
primitive element $u\in\Z_T$, obtained as follows.  For each $\sigma\in\Delta(2)$, we let $\sigma_1:=\sigma +\tau_u$ be an
element in $\Delta_1(3)$.  Then $\Delta_1$ consists of each such $\sigma_1$ and all of its faces.
Since $\Delta_1$ is simplical we may define $h_1 \in\SF(\Delta_1 ,\R)$ on $\sigma_1 =\sigma +\tau_u$ to be $l_{\sigma_1} \in M_{\R}$
with $l_{\sigma_1}|_{\sigma} =0$ and $l_{\sigma_1}(\tau_u)=\epsilon_1 >0$.  Then $h_1$ is strictly upper convex.

Inductively, suppose we have a simplical refinement $\Delta_k$ with a compact strictly upper convex $h_k \in\SF(\Delta_k,\R)$.
Choose an integral $u_{k+1} \in\Int{P}_{\Delta}$ that is not contained in an element of $\Delta_k(1)$.
If $u_{k+1}$ is contained in the interior of
$\sigma_k \in\Delta_k(3)$, then we take the star subdivision of $\sigma_k$ with respect to $\tau_{u_{k+1}}$ spanned by $u_{k+1}$.
This is attained by adding the cones of the form $\sigma_{k+1} =\beta +\tau_{u_{k+1}}$ where $\beta<\sigma_k$ is a proper face.
This gives a refinement $\Delta_{k+1}$.  We set $h_{k+1}$ to be equal to $h_k$ outside the cones in the subdivided $\sigma_k$, and
we define $h_{k+1}$ on the subdivided $\sigma_k$ as follows.
Let $l_{\sigma_k}\in M_{\R}$ define $h_k$ on $\sigma_k$.  Then we define
$l_{\sigma_{k+1}}:= l_{\sigma_k} + m$, where $m\in M_{\R}$ is zero on $\beta$ and $m(u_{k+1})=\epsilon_{k+1}>0$.
Then for sufficiently small $\epsilon_{k+1}>0$ this $h_{k+1}$ is strictly upper convex.

Suppose $u_{k+1}$ is contained in the interior of $\tau_k \in\Delta_k(2)$.  Suppose $\tau_k =\R_{\geq 0} n_1 +\R_{\geq 0} n_2$.
Let $\tau_k^1 =\R_{\geq 0} n_1 +\R_{\geq 0} u_{k+1}$ and $\tau_k^2 =\R_{\geq 0} U_{k+1} +\R_{\geq 0} n_2$.  For
$\sigma_k^1,\sigma_k^2 \in\Delta_k(3)$ that have $\tau_k$ as a face, we have $\sigma_k^i =\tau_k +\alpha^i$ with $\alpha^i\in\Delta_k(1)$,
$i=1,2$.  Then we define $\Delta_{k+1}$ by replacing the cones $\sigma_k^1,\sigma_k^2$ with the cones
$\sigma_{k+1}^{ij}:=\tau_k^j +\alpha^i, i,j =1,2$.  We set $h_{k+1}$ to be equal to $h_k$ outside the cones $\sigma_k^1,\sigma_k^2$.
Let $l_{\sigma_k}^i\in M_{\R}$ define $h_k^i$ on $\sigma_k^i$.
On $\sigma_{k+1}^{ij}$ we define $l_{\sigma_{k+1}^{ij}}:= l_{\sigma_k^i} + m$, where $m\in M_{\R}$ is defined by
$m(n_j)=m(\alpha^i) =0$ and $m(u_{k+1})=\epsilon_{k+1}>0$.  Again, for small enough $\epsilon_{k+1}>0$  $h_{k+1}$ is strictly
upper convex.
\end{proof}

Note that the argument in the proof will also give $h\in\SF(\tilde{\Delta},\Q)$.  Thus the resolution
$X_{\tilde{\Delta}}\rightarrow X_{\Delta}$ is projective.

Theorem~\ref{thm:main-tor} now follows from Proposition~\ref{prop:toric-res} and Corollary~\ref{cor:main}.
It is not difficult to see that the only terminal 3-dimensional Gorenstein toric cones are $\C^3$ and the
quadric hypersurface $X=\{z_0^2 +z_1^2 +z_2^2 +z_3^2 =0\}\subset\C^4$.

\subsection{Quotient singularities}

We recall results on resolutions of quotient singularities $\C^n /G$, where $G$ is a finite group.
For $\C^n /G$ to have trivial dualizing sheaf we must have $G\subset SL(n,\C)$.  We will restrict to the
case of $G$ abelian as this will be sufficient for our purposes.  And in this case we may consider
toric resolutions $\pi :\widehat{\C^n /G}\rightarrow\C^n /G$.

Suppose $G$ is a finite abelian group acting on $\C^n$.  We may assume that the action is diagonal.  And assume the
fixed point set of each nontrivial element has $\codim\geq 2$.  Since $G$ acts freely on $(\C^*)^n$,
$T_{\C}:=(\C^*)^n/G$ is an algebraic torus.  The lattice $\Z ^n \subset\R^n$ is the kernel of
$\exp:\R^n\rightarrow T^n\subset T_{\C}$, where $\exp(x_1,\ldots,x_n)=(e^{2\pi ix_1},\ldots, e^{2\pi ix_n})$.
The group of 1-parameter subgroups of $T_{\C}$ is the lattice $\Z_{T}=\exp^{-1}(G)$.  And $G$ is isomorphic to
$\Z_{T} /\Z^n$.  Then $\C^n/G$ is the toric variety $X_{\Delta}$ associated to the fan $\Delta$ in $\Z_{T}$ given by
the cone $\sigma =\R_{\geq 0} e_1 +\cdots +\R_{\geq 0} e_n$ and all of its faces.  By the condition on the fixed
point set each $e_k \in\Z_T$ is primitive.  The dualizing sheaf of $X_{\Delta}$ is trivial precisely when
the support function $\gamma :\R^n \rightarrow\R$, $\gamma(\sum_{k=1}^n x_k e_k)=\sum_{k=1}^n x_k$, is integral, i.e.
$\gamma(\Z_T)\subseteq\Z^n$.

A resolution of $X_{\Delta}$ is given by a nonsingular subdivision $\tilde{\Delta}$ of $\Delta$.  And as above if
$H_\gamma :=\{x:\gamma(x) =1\}\subset\R^n$, then the triviality of the canonical bundle of $X_{\tilde{\Delta}}$ is equivalent to
each $\tau\in\tilde{\Delta}(1)$ being generated by a primitive element in $P_\Delta \cap\Z_T$, where
$P_\Delta =H_\gamma \cap\Delta$.

Suppose $\pi:X_{\tilde{\Delta}}\rightarrow X_\Delta$ is a crepant resolution.  Then for each $\beta\in\tilde{\Delta}(n)$ we
have
\begin{equation}
 \Vol(\sigma \cap\{x:\gamma(x)\leq 1\}): \Vol(\beta \cap\{x:\gamma(x)\leq 1\})
 =\left[ \Z_T :\Z^n \right] :1= |G|.
\end{equation}
Therefore $|\{\beta: \beta\in\tilde{\Delta}(n)\}| =|G|$.  It is well known~\cite{Ro1} that the Euler characteristic of a toric variety
is given by $\chi(X_{\tilde{\Delta}})=|\{\beta: \beta\in\tilde{\Delta}(n)\}|$.
For $n=2$ and $3$ such a subdivision $\tilde{\Delta}$ always exists, but it may not for $n\geq 4$.

Now suppose that $\C^n/G$ has an isolated singularity, that is, each nonzero element of $G$ only fixes the origin.  Then
we have a version of Proposition~\ref{prop:toric-res} for this situation.
\begin{prop}\label{prop:toric-quot-res}
 Suppose $\C^n/G$ is a Gorenstein isolated singularity, so $G\subset SL(n,\C)$, with $G$ is abelian.  If $n=2$ or $3$, then
 $\C^n/G$ has a projective toric crepant resolution $\pi :\widehat{\C^n /G}\rightarrow\C^n /G$.  Furthermore
 $\chi(\widehat{\C^n /G})=|G|$.
\end{prop}

Note that for $n=2$ and $3$ the Betti numbers of $\widehat{\C^n /G}$ are known.  For $n=3$ we have $b_2 =b_4 =\frac{1}{2}(\chi -1)$.

\section{Toric examples}

\subsection{Toric 3-dimensional examples}

Toric Ricci-flat K\"{a}hler cones are known in abundance.  The first examples of non-regular toric Sasaki-Einstein
manifolds appeared in~\cite{GMSW3} with the metrics given explicitly in~\cite{GMSW1}.  As stated in Theorem~\ref{thm:FOW}
the general existence of Sasaki-Einstein metrics on toric Sasaki manifolds was solved in~\cite{FOW}.  Thus every
$\Q$-Gorenstein toric K\"{a}hler cone admits a Ricci-flat K\"{a}hler cone metric.
See~\cite{CFO} and also~\cite{vC3,vC1} for the construction of infinite series of examples.

The topology of simply connected toric Sasaki-Einstein 5-manifolds is very restricted.  The
first result due to H. Oh~\cite{Oh} determines the homology.
\begin{lem}
Let $S$ be a simply connected 5-manifold with an effective $T^3$-action.  If $S$ has $k$ different
$S^1$ stabilizer subgroups, then $H_2(S,\Z) =\Z^{k-3}$.
\end{lem}

A simply connected Sasaki-Einstein manifold is spin.  Thus the classification of smooth simply connected spin
5-manifolds of S. Smale~\cite{Sma} gives the following.
\begin{thm}
If $S$ is a simply connected toric Sasaki-Einstein 5-manifold with $k$ $S^1$ stabilizer subgroups, then
$S$ is diffeomorphic to $\#(k-3)(S^2 \times S^3)$.
\end{thm}

The first non-regular toric Sasaki-Einstein manifolds were given in a series of examples $S^{p,q}$, with $p,q\in\N, p>q>0$
and $\gcd(p,q)=1$, due to J. Gauntlett, D. Martelli, J. Sparks and D. Waldram~\cite{GMSW1}.
These are all diffeomorphic to $S^2 \times S^3$ and include the first examples of irregular Sasaki-Einstein manifolds.
In~\cite{CFO} and~\cite{vC3,vC1} infinite series of toric Sasaki-Einstein manifolds are constructed.
Together these give infinitely many examples for each $b_2(S)\geq 1$.
From Theorem~\ref{thm:main-tor} we get the following.
\begin{thm}
For each $m\geq 1$, there exist infinitely many toric asymptotically conical Ricci-flat K\"{a}hler manifolds $Y$
asymptotic to a cone over a Sasaki-Einstein structure on $\#m(S^2 \times S^3)$.  For each $m\geq 1$, the
Betti numbers, $b_2(Y)=m+c(X),\ b_4(Y)=c(X)$, of the $Y$ become arbitrarily large.
\end{thm}

\subsection{Resolutions of $C(S^{p,q})$}

We now give some details on the Sasaki-Einstein manifolds $S^{p,q}$, the cones $C(S^{p,q})$, and their
resolutions.  The series $S^{p,q}$, where $p,q\in\N, p>q>0$, and $\gcd(p,q)=1$, first appeared
in~\cite{GMSW3}.  These examples are remarkable in that they contain the first known examples of irregular
Sasaki-Einstein, and also because the metrics are given explicitly (cf.~\cite{GMSW1}).
They appeared as a byproduct of a search for supersymmetric solutions of $D=11$ supergravity.

These examples are diffeomorphic to $S^2 \times S^3$, are toric, and are of cohomogeneity one with an isometry
group of $SO(3)\times U(1)\times U(1)$ if $p,q$ are both odd, and $U(2)\times U(1)$ otherwise.
The Sasaki structure is quasi-regular precisely when $p,q\in\N$ as above satisfy the diophantine equation
\begin{equation}\label{eq:dioph}
4p^2 -3q^2 =r^2,
\end{equation}
for some $r\in\Z$.  It was shown in~\cite{GMSW1} that there are both infinitely many quasi-regular and irregular
examples.

The cone $X_{\Delta} =C(S^{p,q})\cup\{o\}$ is given by the fan $\Delta$ in $\Z^3$ generated by the four vectors
\begin{equation}
u_1 =(0,0,1), u_2 =(1,0,1), u_3 =(p,p,1), u_4 =(p-q-1,p-q,1).
\end{equation}
A basic lattice triangulation of $P_{\Delta}$ can be constructed for general $p,q$ as is shown in Figure~\ref{fig1}
for $C(S^{5,3})$.  We denote the resolved toric manifold by $Y^{p,q}$.  It is not difficult to see that the subdivision
$\tilde{\Delta}$ of $\Delta$ has a compact strictly
convex support function.  Thus Corollary~\ref{cor:main} gives a $p-1$-dimensional family of asymptotically conical
Ricci-flat K\"{a}hler metrics on $Y^{p,q}$.  Note that the crepant resolution $Y^{p,q}$ of $C(S^{p,q})$ is unique.
And by (\ref{eq:dioph}) there are infinitely many examples $Y^{p,q}$ asymptotic to cones over irregular
Sasaki-Einstein manifolds.

\begin{figure}[htb]
 \centering
 \includegraphics[scale=0.4]{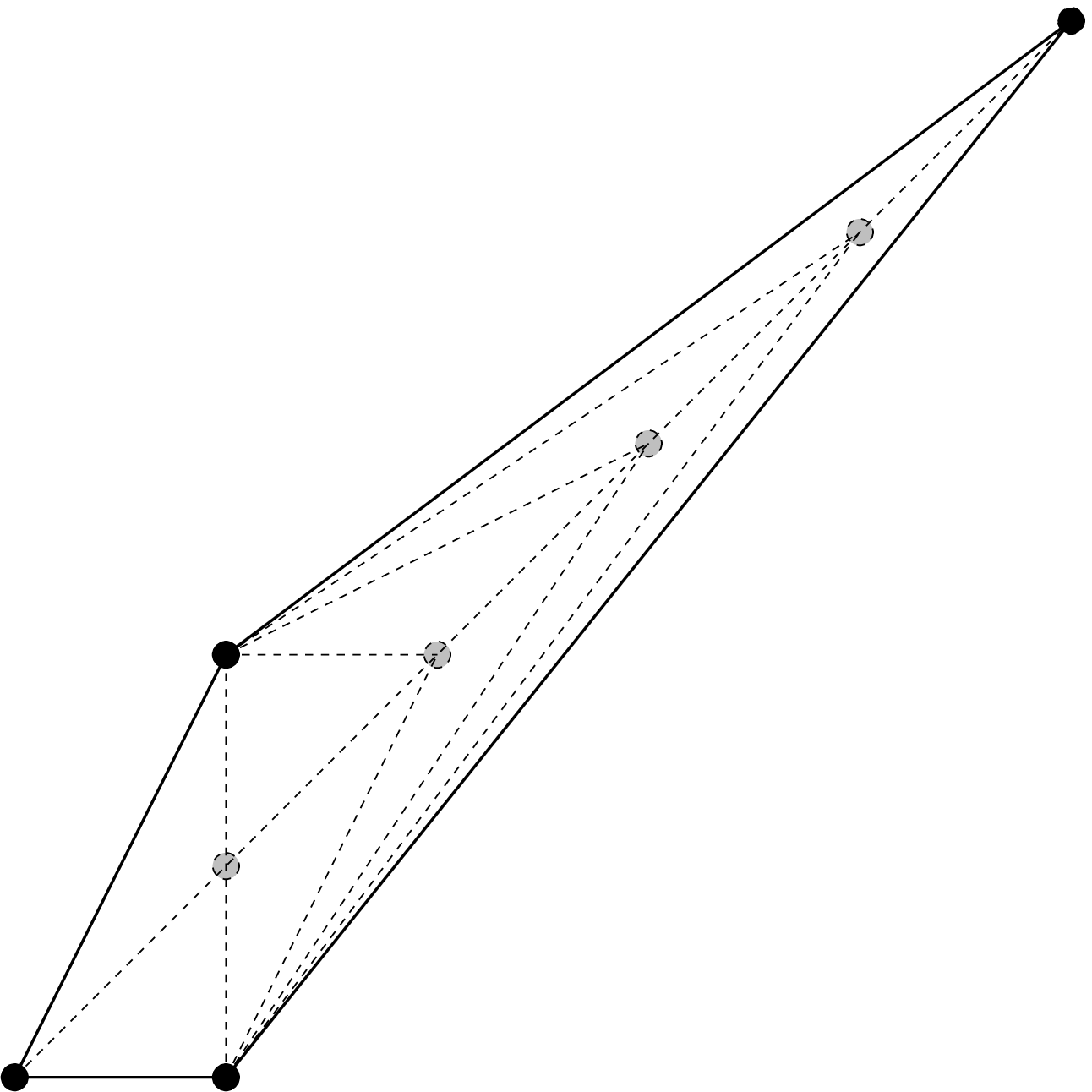}
 \caption{$Y^{5,3}$}
 \label{fig1}
\end{figure}

\section{Hypersurface singularities}

\subsection{Sasaki structures}

We will consider isolated hypersurface singularities defined by quasi-homogeneous polynomials.
Let $\mathbf{w}=(w_0,\ldots, w_n )\in(\Z_+)^{n+1}$ with $\gcd(w_0,\ldots, w_n )=1$.  We have the weighted $\C^*$-action
$\C^*(\mathbf{w})$ on $\C^{n+1}$ given by $(z_0,\ldots,z_n)\rightarrow (\lambda^{w_0}z_0,\ldots,\lambda^{w_n}z_n)$ with weights $w_j$.
A polynomial $f\in\C[z_0,\ldots,z_n]$ is quasi-homogeneous of degree $d\in\Z_+$ if
\begin{equation}
 f(\lambda^{w_0}z_0,\cdots,\lambda^{w_n}z_n) =\lambda^d f(z_0,\ldots,z_n).
\end{equation}
The hypersurface $X_f =\{f=0\}\subset\C^{n+1}$ is called the \emph{weighted affine cone}.  We will assume that the origin
is an isolated singularity.  Then the link
\begin{equation}
 S_f =X_f \cap S^{2n+1},
\end{equation}
where $S^{2n+1}=\{(z_0,\ldots,z_n)\in\C^{n+1} : \sum_{j=0}^n |z_j|^2 =1 \}$ is the unit sphere, is a smooth $(2n-1)$-dimensional
manifold.  Much is known about the topology of $S_f$.  For example, it is proved in~\cite{Mil} that $S_f$ is $(n-2)$-connected.

We have the \emph{weighted} Sasaki structure $(\xi_{\mathbf w}, \eta_{\mathbf w}, \Phi_{\mathbf w}, g_{\mathbf w})$, as in
Example~\ref{ex:Sasak-sph}, on $S^{2n+1}$ for which the Reeb vector field $\xi_{\mathbf w}$ generates the $S^1$-action induced by the above weighted action.  This Sasaki structure will be denoted by $S^{2n+1}_{\mathbf w}$; and, as explained in Example~\ref{ex:Sasak-sph},
the K\"{a}hler cone is $C(S^{2n+1}_{\mathbf w})=\C^{n+1}$ as a complex manifold, but the Euler vector field and potential $r$ are
not the usual ones.

Given a sequence of weights $\mathbf{w}=(w_0,\ldots, w_n )$ as above, we have the graded polynomial ring
$S(\mathbf{w})=\C[z_0,\ldots,z_n]$, where $z_j$ has weight $w_j$.  The \emph{weighted projective space}
$\cps(\mathbf{w})=\cps(w_0,\ldots,w_n)$ is the scheme
$\operatorname{Proj}(S(\mathbf{w}))$ (cf.~\cite{Dol,BeRo}).  Geometrically it is the quotient $(\C^{n+1}\setminus\{0\})/\C^*(\mathbf{w})$.
Equivalently, it is the quotient of $S^{2n+1}_{\mathbf w} =(\xi_{\mathbf w}, \eta_{\mathbf w}, \Phi_{\mathbf w}, g_{\mathbf w})$ by
the weighted circle action $S^1(\mathbf{w})$ given by restricting the $\C^*(\mathbf{w})$-action.  Thus
$\cps(\mathbf{w})$ is a compact complex K\"{a}hler orbifold.

Alternatively, one can use results of~\cite{MSY} to construct a K\"{a}hler cone metric on $\C^{n+1}$ with any Reeb vector field
$Jr\partial_r =\xi=\sum_{j=0}^n b_j\frac{\partial}{\partial\phi_j}$ with $b_j >0,j=0,\ldots,n$, where the $\frac{\partial}{\partial\phi_j}$
generate the $T^{n+1}$-action on $\C^{n+1}$, $(z_0,\ldots,z_n)\rightarrow(e^{i\phi_0}z_0,\ldots,e^{i\phi_n}z_n)$ for
$(\phi_0,\ldots,\phi_n)\in T^{n+1}$.  Then if $(b_0,\ldots,b_n)=(w_0,\ldots,w_n)\in(\Z_+)^{n+1}$, we have a K\"{a}hler cone structure on
$\C^{n+1}$ with Sasaki structure on $S^{2n+1}$ which is equivalent to $S^{2n+1}_{\mathbf{w}}$ up to a transversal K\"{a}hler
deformation.

Now if $f\in\C[z_0,\ldots,z_n]$ is quasi-homogeneous, then $Z_f :=\{[z_0:\cdots:z_n]: f(z_0,\ldots,z_n)=0\}\subset\cps(\mathbf{w})$.
We have the following from~\cite{BG1}.  We will say that $Z_f$ is \emph{quasi-smooth} if the affine cone $X_f$ is smooth outside the origin.
\begin{lem}[\cite{BG1}]\label{lem:hyp-Sasak}
 The link $S_f$ has a quasi-regular Sasaki structure induced from $(\xi_{\mathbf w}, \eta_{\mathbf w}, \Phi_{\mathbf w}, g_{\mathbf w})$
on $S^{2n+1}$ such that we have the commutative diagram
\begin{equation*}
 \begin{CD}
  S_f @>>> S^{2n+1}_{\mathbf{w}} \\
  @VV{\pi}V   @VVV\\
  Z_f @>>> \cps(\mathbf{w})
 \end{CD}
\end{equation*}
where the horizontal maps are Sasakian and K\"{a}hlerian embeddings respectively, and the vertical maps are $S^1$ V-bundle maps and
Riemannian submersions.
\end{lem}
Since we assume that $X_f\subset\C^{n+1}$ has only an isolated singularity at the origin, $Z_f$ is a K\"{a}hler orbifold.
In general we will say that a complex orbifold is \emph{well-formed} if its orbifold singular set has no components of
codimension 1.  Not all of the orbifolds in this article will be well-formed.  When a complex orbifold $Z$ is not well-formed
the orbifold structure is ramified along divisors.  If there is a degree $m_i$ ramification along $D_i\subset Z$, then we call
the $\Q$-divisor $\Delta =\sum_i (1-\frac{1}{m_i})D_i$ the \emph{branch divisor}.  If $K_Z$ denotes the usual canonical class,
then the orbifold canonical class is $K_Z^{\text orb}=K_Z +\Delta$.  Note that the orbifold structure is determined by both the
complex space $Z$ and $\Delta$, so one sometimes writes $(X,\Delta)$ to denote an orbifold.

\begin{prop}[\cite{BGK}]
 The orbifold $Z_f$ is Fano, i.e. the \emph{orbifold} canonical bundle $\mathbf{K}_{Z_f}$ is negative, if and only if
$|\mathbf{w}|=\sum_{j=0}^n w_j >d$.
\end{prop}
Note that the orbifold canonical bundle is different than the usual canonical bundle if $Z_f$ is not well-formed.

It follows that the cone $C(S_f)$ satisfies the condition of Proposition~\ref{prop:CY-cond}.  In fact, by the adjunction formula
the $n$-forms
\begin{equation}\label{eq:hyper-adj}
 \Omega_k :=\frac{(-1)^k}{\partial f/\partial z_k}dz_0 \wedge\cdots\wedge\widehat{dz_k}\wedge\cdots\wedge dz_n |_X ,
\end{equation}
glue together to a global generator of the canonical bundle $\mathbf{K}_{X\setminus\{o\}}$.  The
$\C^*(\mathbf{w})$-action acts on $\Omega$ with weight $\sum_{j=0}^n w_j -d$.  Thus, after possible performing a
$D$-homothetic transformation, we see that $\Omega$ satisfies Proposition~\ref{prop:CY-cond} (iii).

We have the following small generalization of \cite{Re1}, Proposition 4.3.
\begin{prop}
 The hypersurface $X_f\subset\C^{n+1}$ has a rational, and hence canonical singularity at $0$ if and only if
 $|\mathbf{w}|>d$.
\end{prop}
One can check that the form $\Omega$ defined by (\ref{eq:hyper-adj}) satisfies Proposition~\ref{prop:ration} if
and only if $|\mathbf{w}|>d$.

\subsection{Sasaki-Einstein metrics}

We review a sufficient algebro-geometric condition that has been used prolifically to get examples of positive
orbifold K\"{a}hler-Einstein metrics~\cite{DeKol,JoKol}.  This is used for example in the case quasi-homogeneous hypersurfaces
(cf.~\cite{BGK,BG2}) to show that in some cases the Sasaki-structure in Lemma~\ref{lem:hyp-Sasak} has a transversal deformation to a
Sasaki-Einstein structure, but it is not limited to that case (cf.~\cite{Kol3,Kol4}).
This makes use of the following definition.
\begin{defn}\label{defn:klt}
 Let $X$ be a normal complex space and $D\subset X$ a $\Q$-divisor.  Assume that $D$ and $K_X$ are both $\Q$-Cartier.
Let $\pi:Y\rightarrow X$ be a proper birational morphism with $Y$ smooth.  Then there is a unique $\Q$-divisor
$D_Y=\sum_{i}a_i E_i$ such that
\[ K_Y \equiv \pi^*(K_X +D) +D_Y \quad\text{and}\quad\pi_* D_Y=D \]
We say that $(X,D)$ is \emph{Kawamata log terminal(klt)} if the $a_i >-1$ for every $\pi$.
\end{defn}

\begin{thm}[\cite{DeKol,Nad}] \label{thm:K-Econd}
Let $(X,\Delta)$  be an n-dimensional compact complex orbifold with $-K_Z^{\text orb}=-K_Z  -\Delta$ ample.
Suppose there is an $\epsilon>0$ so that
\[ (X, \hbox{$\frac{n+\epsilon}{n+1}$} D +\Delta)\quad\text{ is klt}\]
for every effective $\Q$-divisor $D\equiv -K_Z^{\text orb}$.  Then $(X,\Delta)$ has an orbifold K\"{a}hler-Einstein metric.
\end{thm}
The theorem follows by associating a \emph{multiplier ideal sheaf} to an attempt to solve the Monge-Amp\`{e}re equation for a
K\"{a}hler-Einstein metric, via the continuity method.  The multiplier ideal sheaf is a proper subsheaf of $\mathcal{O}_X$
unless the continuity method produces a K\"{a}hler-Einstein metric.  The condition in Theorem~\ref{thm:K-Econd} guarantees that
this sheaf is $\mathcal{O}_X$.

Checking the condition in Theorem~\ref{thm:K-Econd} is generally quite difficult.  But in certain cases it can be
simplified.  In particular, for perturbations of Brieskorn-Pham singularities simple sufficient numerical criteria are given in
~\cite{BGK} for Theorem~\ref{thm:K-Econd} to be satisfied.

\section{Hypersurface examples}

\subsection{Weighted blow-ups}

We will make use of the notion of a weighted blow-up.  Let $\mathbf{w}=(w_0,\ldots, w_n )$ be a weight vector
and $S(\mathbf{w})=\C[z_0,\ldots,z_n]$ graded polynomial ring as above.  If
$S(\mathbf{w})=\sum_{j\geq 0}S(j)$, where $S(j)$ are the homogeneous elements of degree $j$, and
$f\in S(\mathbf{w})$ is written in homogeneous components $f=\sum_{j\geq 0} f(j)$, then we define the
degree of $f$ to be $w(f)=\min_{j\geq 0} \{f(j)\neq 0\}$.  We have ideals
$M^\mathbf{w}(j)=\{f\in S(\mathbf{w}): w(f)\geq j\}$.
\begin{defn}
Then the \emph{weighted blow-up} $B_0^{\mathbf{w}}\C^{n+1}$
of $\C^{n+1}$ with weight $\mathbf{w}$ is $\operatorname{Proj}(\sum_{j\geq 0} M^\mathbf{w}(j))$.
\end{defn}
For any variety $X\subset\C^{n+1}$ the weighted blow-up $B_0^\mathbf{w} X$ is the birational transform of
$X$ in $B_0^{\mathbf{w}}\C^{n+1}$.
Geometrically $B_0^{\mathbf{w}}\C^{n+1}$ is the total space of the tautological line V-bundle over $\cps(\mathbf{w})$
associated to the $\C^*$-action on $\C^{n+1}\setminus\{0\}$, which has associated $\rank-1$ sheaf
$\mathcal{O}(-1)$.

Alternatively, we may consider $\C^{n+1}$ as the toric variety associated to the fan $\Delta$ consisting of the
cone $\sigma =\R_{\geq 0} e_0 +\R_{\geq 0}e_1 +\cdots +\R_{\geq 0}e_n$ and all of its faces.  Then we have an
integral element $\mathbf{w}\in\Int(\sigma)$.  The weighted blow-up $B_0^{\mathbf{w}}\C^{n+1}$ is then the
toric variety associated to the fan $\Delta^\mathbf{w}$ consisting of the ray $\tau_{\mathbf{w}}$ spanned by
$\mathbf{w}$ and all cones $\tau_{\mathbf{w}} +\beta$ where $\beta<\sigma$ is a proper face.

We give a K\"{a}hler structure on the weighted blow-up $\pi: B_0^{\mathbf{w}}\C^{n+1}\rightarrow\C^{n+1}$.  The
Sasaki structure $S^{2n+1}_{\mathbf{w}}$ as in Example~\ref{ex:Sasak-sph} has the K\"{a}hler cone $C(S^{2n+1}_{\mathbf w})$ which
is biholomorphic to $\C^{n+1}$ and has metric $dr^2 +r^2 g$ with K\"{a}hler form $\omega =dd^c\frac{r^2}{2}$.  Note that
$r$ is not the usual radius function on $\C^{n+1}$, and $r^2$ is not necessarily smooth at $o\in\C^{n+1}$.  But $\pi^* r^2$ is
smooth on $B_0^{\mathbf{w}}\C^{n+1}$.  Let $\phi:\R_{\geq 0}\rightarrow[0,1]$ be a smooth function with $\phi(r)=1$ for
$r\leq a$ and $\phi(r)=0$ for $r\geq b$ for $0<a<b$.  Then $\omega_0 =dd^c(\phi\log r^2) +Cdd^c r^2$ defines, an orbifold, K\"{a}hler form
on $B_0^{\mathbf{w}}\C^{n+1}$ for sufficiently large $C>0$.  And we have $[\omega_0]\in H_c^2(Y,\R)$.

Let $F$ be the exceptional divisor of $\pi:B_0^{\mathbf{w}}\C^{n+1}\rightarrow\C^{n+1}$.
For a hypersurface $X =\{f=0\}\subset\C^{n+1}$ the weighted blow-up is the birational transform
$X'=B_0^{\mathbf{w}}X\subset B_0^{\mathbf{w}}\C^{n+1}$ with exceptional divisor $E=F\cap X'$.
We have the following adjunction formula~\cite{Re2}
\begin{equation}\label{eq:blow-up-adj}
K_{X'} = \pi^*K_X +(w(z_0 \cdots z_n)-w(f)-1)E|_{X'}.
\end{equation}

\subsection{Examples with $b_3 \neq 0$}

We consider the simplest cases in which the terminalization procedure of M. Reid (cf.~\cite{Re1,Re2})
produces a smooth crepant resolution $\pi:Y\rightarrow X$.
This says that one can construct a projective crepant resolution $Y$ of a 3-fold $X$ with only
canonical singularities such that $Y$ has only terminal singularities.
If $X$ is Gorenstein, then one successively resolves the isolated non-terminal singularities with blow-ups
of weights $(1,1,1,1)$, $(2,1,1,1)$, or $(3,2,1,1)$.

\begin{figure}[tbh]\label{fig:b3neq0}
\centering
\begin{tabular}{|c|c|c|c|c|c|c|}
\hline
$X$ & $S$ S-E & \multicolumn{2}{|c|}{crepant $Y$} & $c(X)$ & $b_3(Y)$ &  $H_2(S)$ \\\hline\hline
\multirow{3}{*}{ $x_0^3 +x_1^3 +x_2^3 +x_3^k =0$ } & \multirow{3}{12mm}{$k=3$, $k>6$} & $0$ & yes & $\lfloor\frac{k}{3}\rfloor$ & $2(\lfloor\frac{k}{3}\rfloor-1)$ &  $\Z^6 \oplus\Z_{\frac{k}{3}}^2$ \\\cline{3-7}
 & & $1$ & yes & $\lfloor\frac{k}{3}\rfloor$ & $2\lfloor\frac{k}{3}\rfloor$ & $\Z_k ^2$ \\\cline{3-7}
 & & $2$ & no & & & $\Z_k^2$\\\hline
\multirow{3}{*}{ $x_0^2 +x_1^4 +x_2^4 +x_3^k =0$ } & \multirow{3}{12mm}{$k=4$, $k>10$} & $0$ & yes & $\lfloor\frac{k}{4}\rfloor$ & $2(\lfloor\frac{k}{4}\rfloor -1)$ & $\Z^7 \oplus\Z_{\frac{k}{4}}^2$\\\cline{3-7}
 & & $1$ & yes & $\lfloor\frac{k}{4}\rfloor$ & $2\lfloor\frac{k}{4}\rfloor$ & $\Z_k^2$\\\cline{3-7}
 & & $2$ & unknown &   & & $\Z^3 \oplus\Z_{\frac{k}{2}}^2$ \\\cline{3-7}
 & & $3$ & no &   & & $\Z_k^2$\\\hline
\multirow{3}{*}{ $x_0^2 +x_1^3 +x_2^6 +x_3^k =0$ } & \multirow{3}{12mm}{$k=6$, $k>12$} & $0$ & yes & $\lfloor\frac{k}{6}\rfloor$ & $2(\lfloor\frac{k}{6}\rfloor-1)$ & $\Z^8 \oplus\Z_{\frac{k}{6}}^2$\\\cline{3-7}
 & & $1$ & yes & $\lfloor\frac{k}{6}\rfloor$ & $2\lfloor\frac{k}{6}\rfloor$ & $\Z_k^2$\\\cline{3-7}
 & & $2$ & unknown &   & & $\Z^2 \oplus\Z_{\frac{k}{2}}^2$\\\cline{3-7}
 & & $3$ & unknown &   & & $\Z^4 \oplus\Z_{\frac{k}{3}}^2$\\\cline{3-7}
 & & $4$ & unknown &   & & $\Z^2 \oplus\Z_{\frac{k}{2}}^2$\\\cline{3-7}
 & & $5$ & no      &   & & $\Z_k^2$      \\\hline
$x_0^3 +x_1^4 +x_2^4 +x_3^4 =0$ & yes & \multicolumn{2}{|c|}{yes} & 3 & 12 & $\Z_3^6$ \\\hline
\end{tabular}
\caption{Examples with $b_3(Y)\neq 0$}
\end{figure}

The first example in Figure~\ref{fig:b3neq0} for $k\geq 3$ admits a crepant blow-up with weight
$(1,1,1,1)$.  The result is smooth besides one singularity of the form $x_0^3 +x_1^3 +x_2^3 +x_3^{k-3} =0$.
Proceeding inductively, one gets a smooth resolution $Y$ if $k=0$ or $1 \mod 3$.  The second example
for $k\geq 4$ admits a crepant blow-up with weight $(2,1,1,1)$.  The blow-up has one singularity of the form
$x_0^2 +x_1^4 +x_2^4 +x_3^{k-4} =0$.  And for $k=0$ or $1 \mod 4$ after repeating this blow-up we get a
smooth resolution.  The third example is completely analogous but we repeatedly blow-up with weight
$(3,2,1,1)$, and for $k=0$ or $1 \mod 6$ we get a smooth resolution.  The resolutions of the these three series
of hypersurfaces is due to H.-W. Lin~\cite{Lin}.  For the fourth example, the usual
blow-up is crepant.  The result has a smooth genus 3 curve of $A_2$ singularities.  Blowing up along
this curve results in a smooth resolution.

Since each blow-up admits a compact K\"{a}hler class, Theorem~\ref{thm:main}
proves that each resolved space $Y$ admits a $c(X)$ dimensional space of asymptotically conical Ricci-flat K\"{a}hler
metrics for the range given in the second column for which $S$ is known to admit a Sasaki-Einstein metric.

The second column gives the range of $k$ for which the Sasaki link
$S$ is known to admit a Sasaki-Einstein metric.  This can be proved from the simple numerical condition in Theorem 34 of~\cite{BGK} for examples
which are perturbations of Brieskorn-Pham singularities.  The third column gives $k$ mod $3$, $4$ and $6$, the fourth says whether the $X$ admits a crepant resolution $\pi:Y\rightarrow X$, the fifth gives the number $c(X)$ of prime divisors in
$\pi^{-1}(o)= E$, and the sixth gives the third Betti number of the resolution $Y$.

The homology of $S$ can be computed using well known
results on hypersurface singularities (cf.~\cite{MilOr}).  But a simpler method is to use a result on the homology of a 5-dimensional Seifert
bundle over a complex orbifold of J. Koll\'{a}r~\cite{Kol3}.  Let $(Z,\Delta)$ be an orbifold with branch divisor $\Delta =\sum_i (1-\frac{1}{m_i})D_i$,
whose singularities are all locally quotients of cyclic groups.  Suppose $\pi: S\rightarrow Z$ is a Seifert $S^1$-bundle with smooth total space.
So $c_1(S/Z)\in H^2(Z,\Q)$.  Let $\Ord(Z,\Delta)$ be the l.c.m. of the orders of the local groups of the orbifold.  Then
$\Ord(Z,\Delta)c_1(S/Z)\in H^2(Z,\Z)$, and let $d\in\N$ be the greatest number dividing this integral class.
\begin{thm}[\cite{Kol3}]\label{thm:cohom-seif}
Suppose $\pi: S\rightarrow Z$ is a smooth Seifert $S^1$-bundle over a projective orbifold.  Assume $H_1(S,\Q)=0$ and $H_1^{\text orb}(Z,\Z)=0$.
If $s=\rank H^2(Z,\Q)$, then $H^2(S,\Z)$ is as follows.

\centering
\begin{tabular}{|c|c|c|c|c|c|}
\hline
$H^0$ & $H^1$ & $H^2$ & $H^3$ & $H^4$ & $H^5$ \\\hline
$\Z$ & $0$ & $\Z^{s-1} \oplus\Z_d$ & $\Z^{s-1}\oplus\sum_i \Z_{m_i}^{2g(D_i)}$ & $\Z_d$ & $\Z$ \\\hline
\end{tabular}
\end{thm}
As in Lemma~\ref{lem:hyp-Sasak} the Sasaki link is a Seifert bundle $\pi: S\rightarrow Z$ over a weighted homogeneous hypersurface.
We have $\pi_1(S)=e$ as the link of an $n$-dimensional hypersurface singularity is $(n-2)$-connected~\cite{Mil}.
We consider the Milnor algebra to compute $\rank H^2(Z,\Q)$.
\begin{equation}
 M(f) =\frac{\C[x_0,\ldots,x_n]}{(\partial f/\partial x_0,\ldots,\partial f/\partial x_n)}
\end{equation}
Then $M(f)$ is a graded algebra, and we denote by $M(f)_n$ the degree $n$ homogeneous component.
\begin{thm}[\cite{Ste}]\label{thm:cohom-hyp}
 The Hodge numbers of the primitive cohomology $H^n_0(Z)$ of an $n$-dimensional, degree $d=w(f)$, quasi-smooth homogeneous hypersurface
$Z_f \subset\cps^{n+1}(\mathbf{w})$ are given by
\[ h_0^{i,n-i} =\dim_{\C}M(f)_{(i+1)d-|\mathbf{w}|}.\]
\end{thm}
It is not difficult to compute $H_2(S)$ from Theorem~\ref{thm:cohom-hyp} and Theorem~\ref{thm:cohom-seif}.  Note that since $S$ is simply connected and spin, this completely determines the diffeomorphism type of $S$ from the results of S. Smale~\cite{Sma}.

The resolution procedure for $k=2 \mod 3$ in the first family, $k=3 \mod 4$ in the second, and $k=5 \mod 6$ in the third stops with a $Y$ with a single terminal singularity.  It follows from (\ref{eq:Flen}) that these singularities are analytically $\Q$-factorial.
It follows from Theorem~\ref{thm:can-3fold} (iii) that any partial crepant resolution must be singular.

For the first three examples the divisors $E_i,i=1,\ldots,c(X)-1$ in $E=\pi^{-1}(o)$ besides the last, are ruled surfaces over an elliptic curve.  For $k=0 \mod 3, 4,$ and $6$ the divisor $E_c$ is a del Pezzo surface.  And for $k=1 \mod 3, 4,$ and $6$ the divisor $E_c$ is a cone over an elliptic curve.

\subsection{Resolutions of quotient singularities}

It is well known~\cite{Ro2} that every Gorenstein quotient singularity in dimensions $n=2$ and $n=3$ has a crepant resolution.
That is, for a finite group $G\subset SL(n,\C)$, there is a crepant resolution of $\C^n /G$ for $n=2$ and $n=3$.
Therefore if $\dim X\leq 3$ has a partial crepant resolution $\pi:Y\rightarrow X$ so that $Y$ has only orbifold singularities,
then $Y$ can be resolved to get a crepant resolution of $X$, $\tilde{\pi}:\tilde{Y}\rightarrow X$.
For our purposes we will only need to consider abelian, in fact cyclic, groups $G$.  Thus the toric resolutions in Section~\ref{subsec:tor-res} are sufficient.

Suppose $X=C(S)$ is a K\"{a}hler cone satisfying Proposition~\ref{prop:CY-cond} such that $S$ is quasi-regular with leaf space $Z$.
Then $C(S)=(\mathbf{K}_Z^\times)^{\frac{p}{q}}$, for $p,q\in\Z_+$, where $\mathbf{K}_Z^\times$ denotes the total space of the orbifold
canonical line bundle minus the zero section.  If $p=q=1$, then the total space of $\mathbf{K}_Z$ provides an obvious partial crepant resolution of $X$
with orbifold singularities.  Since $\mathbf{K}_Z$ has a negative curvature connection, given by the contact structure
$\eta$ on $S$, there is a bimeromorphic map $\pi: Y=\mathbf{K}_Y \rightarrow X$ given by collapsing the zero section~\cite{Gra}.

Let $\Picorb Z$ be the Picard group of orbifold line bundles on $Z$.  Since $\mathbf{K}_Z$ is negative,
standard arguments show that $\Picorb Z \cong H_{orb}^2 (Z,\Z)$.
\begin{defn}
The \emph{index} of a Fano orbifold $Z$, denoted $\Ind Z$, is the largest positive integer $m$ such that
$\frac{1}{m}c_1(\mathbf{K}_Z)\in H_{orb}^2 (Z,\Z)$.  Equivalently, $m$ is the largest positive integer such
that $\mathbf{K}_Z$ admits an $m$\superscript{th} root $\mathbf{K}_Z^{\frac{1}{m}}$.
\end{defn}

If $S$ is simply connected and $\Ind(Z)=1$, then $C(S)=\mathbf{K}_Z^\times$ and we have a partial resolution
$\pi: Y \rightarrow X$.  This is the case in the examples of Figure~\ref{fig:b3eq0}.  These Sasaki-Einstein manifolds first appeared in
~\cite{JoKol} and~\cite{BGN1}.  They the examples whose leaf spaces are the anti-canonically embedded quasi-smooth and well formed 2-dimensional hypersurfaces. These log del Pezzo hypersurfaces were classified in~\cite{JoKol}, and are listed in Figure~\ref{fig:b3eq0}.  They are all proved to admit K\"{a}hler-Einstein metrics.  Most were proved to admit K\"{a}hler-Einstein metrics in~\cite{JoKol}.
The remaining cases were proved in~\cite{BGN1}
for weights $(2,3,5,9)$, in~\cite{BGN2} for weights $(1,3,5,8)$, and in~\cite{Ara} for weights $(1,2,3,5)$ and $(1,3,5,7)$.

Since $Z_f\subset\cps(\mathbf{w})$ with $\mathbf{w}=(w_0,w_1,w_2,w_3)$ is anti-canonically embedded, the total space of $\mathbf{K}_{Z_f}$ is isomorphic to the weighted blow-up $X'=B_0^\mathbf{w} X_f$ of the quasi-homogeneous hypersurface $X_f=\{f=0\}\subset\C^4$.  Since $Z_f$ is well formed it, only has isolated singularities.  If $x\in Z_f$ is a singular point with local group $\Z_p$ acting with weights $(r,s)\in\Z^2$, i.e. $(z_1,z_2)\rightarrow (\alpha^r z_1,\alpha^s z_2)$, $\alpha\in\Lambda_p$ the p-th roots of unity, then $x\in X'$ is an orbifold point with weights $(r,s,-r-s)$.  And $X'$ has only isolated singularities.  We can construct a smooth resolution $\pi:Y\rightarrow X'\rightarrow X_f$ by gluing the resolutions of Proposition~\ref{prop:toric-quot-res}.

\begin{figure}[tbh]\label{fig:b3eq0}
\centering
\begin{tabular}{|c|c|c|c|c|c|}
\hline
$\mathbf{w}=(w_0,w_1,w_2,w_3)$ & $d$ & $c(X)$ & $m_{\mathbf{w}}$ & $n_{\mathbf{w}}$ & $S$ \\\hline\hline
$(2,2k+1,2k+1,4k+1)$ & $8k+4$ & $6k+1$ &  12 & 5 & $\#7(S^2\times S^3)$ \\\hline\hline
$(1,2,3,5)$          & 10     &  2     &  17 & 5 & $\#8(S^2\times S^3)$ \\\hline
$(1,3,5,7)$          & 15     &  4     &  19 & 8 & $\#8(S^2\times S^3)$ \\\hline
$(1,3,5,8)$          & 16     &  4     &  20 & 8 & $\#9(S^2\times S^3)$ \\\hline
$(2,3,5,9)$          & 18     &  3     &  13 & 5 & $\#6(S^2\times S^3)$ \\\hline
$(3,3,5,5)$          & 15     &  12    &  10 & 2 & $\#4(S^2\times S^3)$ \\\hline
$(3,5,7,11)$         & 25     &  10    &  8  & 3 & $\#4(S^2\times S^3)$ \\\hline
$(3,5,7,14)$         & 28     &  10    &  9  & 4 & $\#5(S^2\times S^3)$ \\\hline
$(3,5,11,18)$        & 36     &  10    &  10 & 3 & $\#5(S^2\times S^3)$ \\\hline
$(5,14,17,21)$       & 56     &  24    &  5  & 1 & $\#3(S^2\times S^3)$ \\\hline
$(5,19,27,31)$       & 81     &  27    &  5  & 1 & $\#2(S^2\times S^3)$ \\\hline
$(5,19,27,50)$       & 100    &  25    &  6  & 1 & $\#3(S^2\times S^3)$ \\\hline
$(7,11,27,37)$       & 81     &  27    &  4  & 0 & $\#2(S^2\times S^3)$ \\\hline
$(7,11,27,44)$       & 88     &  27    &  6  & 1 & $\#3(S^2\times S^3)$ \\\hline
$(9,15,17,20)$       & 60     &  14    &  4  & 0 & $\#2(S^2\times S^3)$ \\\hline
$(9,15,23,23)$       & 69     &  46    &  7  & 0 & $\#4(S^2\times S^3)$ \\\hline
$(11,29,39,49)$      & 127    &  44    &  4  & 0 & $\#2(S^2\times S^3)$ \\\hline
$(11,49,69,128)$     & 256    &  64    &  4  & 0 & $S^2\times S^3$       \\\hline
$(13,23,35,57)$      & 127    &  63    &  4  & 0 & $\#2(S^2\times S^3)$ \\\hline
$(13,35,81,128)$     & 256    &  64    &  4  & 0 & $S^2\times S^3$      \\\hline
\end{tabular}
\caption{Examples with $b_3(Y)=0$}
\end{figure}

We have one series of examples and the rest are sporadic.  For each possible set of weights $(w_0,w_1,w_2,w_3)$ we give the number of
exceptional divisors $c(X)$; $m_{\mathbf{w}}$ is the complex dimension of the space of admissible polynomials, i.e. those of weighted degree $d$
giving a quasi-smooth hypersurface; $n_{\mathbf{w}}$ is the dimension of $m_{\mathbf{w}}$ modulo the action of the automorphism group of
$\cps(\mathbf{w})$, thus $n_{\mathbf{w}}$ is the complex dimension of the moduli of Sasaki-Einstein structures; the last column gives the
link $S$ up to diffeomorphism.  These deformations preserve the types of the singularities.  Thus we have moduli of Ricci-flat K\"{a}hler
asymptotically conical manifolds.  Since $b_3(Y)=0$, the Betti numbers of $Y$ are determined by the information in the table.

\subsection{Higher dimensional examples}

The first example of Figure~\ref{fig:b3neq0} easily generalizes to arbitrary dimensions.  We consider
the hypersurface $X_k =\{x_0^n +x_1^n +\cdots +x_{n-1}^n +x_n^k =0\}\subset\C^{n+1}$ with $k\geq n$.  we see from
(\ref{eq:blow-up-adj}) that the usual blow-up at $o\in X_k$ is crepant.  Then $B_0 X_k$ has one singularity
isomorphic to $X_{k-n}$.  Proceeding inductively, we get a smooth crepant resolution $Y_k$ if $k= 0$ or $1 \mod n$,
since the surface $X_k$ for $k=0$ and $1$ is smooth.

\begin{figure}[tbh]\label{fig:hdim}
\centering
\begin{tabular}{|c|c|c|c|}
\hline
$X$ & $S$ S-E & $k \mod n$ & $c(X)$  \\\hline
\multirow{2}{*}{ $x_0^n +x_1^n +\cdots +x_{n-1}^n +x_n^k =0$} & \multirow{2}{*}{$k>n(n-1)$, $k=n$} & 0 & $\lfloor\frac{k}{n}\rfloor$  \\\cline{3-4}
 & & 1 & $\lfloor\frac{k}{n}\rfloor$  \\\hline
\end{tabular}
\caption{Examples in higher dimensions}
\end{figure}

We list some of the properties of the resolved spaces $Y_k$ in Figure~\ref{fig:hdim}.  The second column gives the
range of $k$ for which the Sasaki link $S$ is known to admit a Sasaki-Einstein metric from the numerical criteria
in~\cite{BGK}.  Recall that $S$ is $(n-2)$-connected, so the only non-trivial homology is in dimensions $n-1$ and $n$.
The non-trivial Betti numbers are given via a formula in
~\cite{MilOr}.  For $k= 0 \mod n$ we have
\begin{equation}
b_{n-1}(S)=(-1)^{n+1}\left(1+\frac{(1-n)^{n+1} -1}{n}\right).
\end{equation}
For $k=1 \mod n$, $S$ is a rational homology sphere and the order of its homology is given by Theorem 3 of
~\cite{BG3}.
\begin{equation}
|H_{n-1}(S)|=k^{b_{n-2}},\quad\text{where } b_{n-2}=(-1)^n\left(1+\frac{(1-n)^n -1}{n}\right).
\end{equation}
Here $b_{n-2}$ is the Betti number of link of the Calabi-Yau hypersurface
$F=\{x_0^n +\cdots +x_{n-1}^n =0\}\subset\cps^{n-1}$.

The exceptional divisors of the resolution $\pi:Y_k \rightarrow X_k$ are ruled varieties
$E_j =\mathbb{P}(\mathcal{O}_F(1)\oplus\mathcal{O}_F), j=1,\ldots, c(X)-1$, besides the last which for
$k=0 \mod n$ is the Fano hypersurface
$E_{c} =\{x_0^n +x_1^n +\cdots +x_{n-1}^n +x_n^n =0\}\subset\cps^n$ and for $k=1 \mod n$ is the cone over $F$,
$E_{c} =\{x_0^n +x_1^n +\cdots +x_{n-1}^n  =0\}\subset\cps^n$.
The Euler characteristic of $Y_k$ can be easily computed
\begin{equation}
 \chi(Y_k)=
   \begin{cases}
    \frac{c}{n}\left((1-n)^n -1\right) +cn -(1-n)^n +1 & \text{if }k=0 \mod n, \\
    \frac{c}{n}\left((1-n)^n -1\right) +cn +1          & \text{if }k=1 \mod n,
   \end{cases}
\end{equation}
where $c=c(X)$.

\bibliographystyle{plain}

\end{document}